\documentclass[12pt]{article}

\usepackage[utf8]{inputenc}
\usepackage[english]{babel}
\usepackage[T1]{fontenc}
\usepackage{amsmath}
\usepackage{amsfonts}
\usepackage{amssymb}
\usepackage{amsthm}
\usepackage{hyperref}
\usepackage{cleveref}
\usepackage{thmtools}
\usepackage{thm-restate}
\usepackage{enumerate}
\usepackage{subcaption}
\usepackage{todonotes}
\usepackage{tabto}
\usepackage{dsfont}
\usepackage{xspace}
\usepackage{relsize}
\usepackage{tikz}
	\usetikzlibrary{backgrounds,calc,decorations.pathmorphing}
	\tikzset{every path/.style={thick}}
	\pgfdeclarelayer{bg}    
	\pgfsetlayers{bg,main}  
	
\usepackage[centering,
			textwidth=15cm,
			textheight=22cm,
			top=3.5cm,
			footskip=40pt,
			marginparwidth=2.9cm
			]{geometry}

\numberwithin{equation}{section}

\theoremstyle{plain}
\newtheorem{thm}{Theorem}[section]
\newtheorem{prop}[thm]{Proposition}
\newtheorem{lem}[thm]{Lemma}
\newtheorem{cor}[thm]{Corollary}

\newtheorem{obs}[thm]{Observation}

\theoremstyle{definition}
\newtheorem{defi}[thm]{Definition}
\newtheorem{ex}[thm]{Example}
\newtheorem{question}{Open Problem}

\makeatletter
\newcommand*\rel@kern[1]{\kern#1\dimexpr\macc@kerna}
\newcommand*\widebar[1]{%
	\begingroup
	\def\mathaccent##1##2{%
		\rel@kern{0.8}%
		\overline{\rel@kern{-0.8}\macc@nucleus\rel@kern{0.2}}%
		\rel@kern{-0.2}%
	}%
	\macc@depth\@ne
	\let\math@bgroup\@empty \let\math@egroup\macc@set@skewchar
	\mathsurround\z@ \frozen@everymath{\mathgroup\macc@group\relax}%
	\macc@set@skewchar\relax
	\let\mathaccentV\macc@nested@a
	\macc@nested@a\relax111{#1}%
	\endgroup
}
\makeatother

\newcommand{\calE}{\mathcal{E}}
\newcommand{\R}{\mathbb{R}}
\newcommand{\Z}{\mathbb{Z}}
\newcommand{\dotcup}{\hspace{.22em}\ensuremath{\mathaccent\cdot\cup}\hspace{.22em}}
\newcommand{\conv}{\mathrm{conv}}
\newcommand{\cone}{\mathrm{cone}}
\newcommand{\CutP}{\textsc{Cut}}

\newcommand{\Bond}{\textsc{Bond}}
\newcommand{\supp}{\mathrm{supp}}
\newcommand{\mypar}{\paragraph}
\newcommand{\qq}[1]{``#1''}
\newcommand{\MC}{\textsc{MaxCut}\xspace}
\newcommand{\MB}{\textsc{MaxBond}\xspace}

\newcommand{\V}{{\widebar{V}}}
\newcommand{\E}{{\widebar{E}}}
\newcommand{\G}{{\widebar{G}}}
\renewcommand{\a}{{\widebar{a}}}

\newcommand{\ov}{\widebar{v}}
\newcommand{\prism}{\textit{Prism}}
\newcommand{\relint}{\mathrm{relint}}
\newcommand{\e}{{e^*}}
\newcommand{\bnull}{{\bf 0}}

\providecommand{\keywords}[1]{\textbf{\textit{keywords---}} #1}

\title{On the Bond Polytope}
\author{Markus Chimani, Martina Juhnke-Kubitzke, Alexander Nover}
\date{\smaller School of Mathematics/Computer Science, Uni Osnabrück, Germany\\
\{markus.chimani,juhnke-kubitzke,alexander.nover\}@uni-osnabrueck.de}

\begin{document}
\maketitle
\begin{abstract}
	Given a graph $G=(V,E)$, the maximum bond problem searches for a maximum cut $\delta(S) \subseteq E$ with $S \subseteq V$ such that $G[S]$ and $G[V\setminus S]$ are connected. This problem is closely related to the well-known maximum cut problem and known under a variety of names such as largest bond, maximum minimal cut and maximum connected (sides) cut.
	The bond polytope is the convex hull of all incidence vectors of bonds. Similar to the connection of the corresponding optimization problems, the bond polytope is closely related to the cut polytope. While cut polytopes have been intensively studied, there are no results on bond polytopes. We start a structural study of the latter.
	
	 We investigate the relation between cut- and bond polytopes and study the effect of graph modifications on bond polytopes and their facets. Moreover, we study facet-defining inequalities arising from edges and cycles for bond polytopes. In particular, these yield a complete linear description of bond polytopes of cycles and $3$-connected planar $(K_5-e)$-minor free graphs. Moreover we present a reduction of the maximum bond problem on arbitrary graphs to the maximum bond problem on $3$-connected graphs. This yields a linear time algorithm for maximum bond on $(K_5-e)$-minor free graphs.
\end{abstract}

\section{Introduction}
The problem of finding a maximum cut in a weighted graph, called \MC problem, is well-known in combinatorial optimization, and one of Karp's original 21 NP-complete problems~\cite{karp}. The maximum bond problem (\MB) is obtained from this by adding a connectivity requirement for both sides of the cut.

Formally, considering a graph $G=(V,E)$ with edge weights $c_e$, \MC is the problem of finding a node subset $S\subseteq V$ that 
maximizes $\sum_{e \in \delta(S)}c_e$, where $\delta(S)=\{e \in E: |e \cap S|=1\}$.
Such a set $\delta(S)$ is called a \emph{cut}. It is a \emph{bond}, if both sides of the cut, i.e., $G[S]$ and $G[V \setminus S]$, are connected. \MB is the problem of finding a bond $\delta(S)$ maximizing $\sum_{e \in \delta(S)}c_e$.
This problem is known under a variety of names including maximum minimal cut \cite{parameterized_maximum_cut_connectivity}, largest bond \cite{computing_largest_bond}, connected max cut \cite{bond_K5-e_free}, and maximum connected sides cut problem \cite{bond_series_parallel}.
To avoid confusion with the maximum \emph{one-sided} connected cut problem \cite{MaxBondMaxConnectedCut,parameterized_maximum_cut_connectivity,DBLP:journals/ipl/GandhiHKPS18,DBLP:conf/esa/HajiaghayiKMPS15,DBLP:journals/tcs/HajiaghayiKMPS20} we stick to the naming maximum bond.
The research on \MB is driven by applications like image segmentation~\cite{application_image}, forest planning~\cite{application_forest}, and computing market splittings~\cite{application_market}.

\MB is known to be NP-complete \cite{Approx_Intractability_MaxCut_Vartiants}, even when restricted to $3$-connected planar graphs \cite{Approx_Intractability_MaxCut_Vartiants} or bipartite planar graphs \cite{computing_largest_bond, MaxBondMaxConnectedCut,parameterized_maximum_cut_connectivity}. 
Conversely, \MB is solvable in linear time on series-parallel graphs \cite{bond_series_parallel}.
Moreover there is an extensive study of the parameterized complexity of \MB \cite{computing_largest_bond, MaxBondMaxConnectedCut,parameterized_maximum_cut_connectivity}.
On the other hand, it is known that there is no constant factor approximation (if $\text{P} \neq \text{NP}$) \cite{computing_largest_bond,MaxBondMaxConnectedCut}.

Besides the mentioned algorithmic results there is only little knowledge on the maximum bond in a general graph:
Ding, Dziobiak and Wu proved that the maximum bond in any simple $3$-connected graph $G$ with $|V(G)|=n$ has size at least $\frac{2}{17}\sqrt{\log n}$ and conjectured that the maximum bond in such a graph has size $\Omega(n^{\log_3 2})$ \cite{Bond_3connected}.
This conjecture was verified by Flynn for several graph classes including planar graphs \cite{VerifiedConjecture} but remains open in general.

In this work we consider \MB from a polyhedral viewpoint. To this end we introduce the \emph{bond polytope} $\Bond(G)$ which is closely related to the intensively studied cut polytope $\CutP(G)$ \cite{bipartitesubgraph,OnTheCutPolytope,DL1,DL2,GeometryOfCutsAndMetrics,doi:10.1111/itor.12194,GeneralizedCutAndMetricPolytopes}.
The bond polytope (resp. cut polytope) is defined as the convex hull of the indicator vectors $x^\delta$ of all bonds (resp. cuts) $\delta$ in $G$, given by 
$$x^{\delta}_e= \begin{cases}	1, &\text{ if } e \in \delta,\\
0, &\text{ else}.			\end{cases}$$

A big part of the research on cut polytopes is the investigation of facet-defining inequalities for specific graph classes such as complete graphs or circulants, see, e.g., \cite{DL1,DL2,GeometryOfCutsAndMetrics,POLJAK1992379} for an overview. 
This is motivated by the fact that such inequalities give rise to valid inequalities for all graphs containing their support graph.
Moreover, there is an extensive study of the effect of graph modifications (such as node splitting and edge contractions) on cut polytopes and their facets \cite{OnTheCutPolytope}.
Complete linear descriptions of cut polytopes are known for the classes of graphs obtained by excluding $K_5$ or $K_{3,3}$ as a minor \cite{OnTheCutPolytope,CutPolytope_K33}. For $K_5$-minor free graphs this description is given by inequalities associated to edges and cycles \cite{OnTheCutPolytope}. Considering the polytope defined by these inequalities for an arbitrary graph $G$, we obtain the (semi-)metric polytope, a  well studied relaxation of $\CutP(G)$ \cite{DL1,DL2,GeometryOfCutsAndMetrics,doi:10.1111/itor.12194,GeneralizedCutAndMetricPolytopes}.

In this paper, we start the structural study of bond polytopes.

\mypar{Our contribution and organization of this paper.}
	After recalling some basic definitions in \Cref{sec:Preliminaries}, 
	we discuss the relation of cut- and bond polytopes in \Cref{sec:FirstProperties}. This includes the observation that several fundamental properties of cut polytopes do not carry over to bond polytopes.
	
	In \Cref{sec:ConstructingFacetsFromFacets} we study how graph modifications (such as node splitting and edge contraction) effect bond polytopes and their facets.
	
	In \Cref{sec:Reduction} we present an efficient (linear-time) reduction of \MB on arbitrary graphs to \MB on $3$-connected graphs. This algorithm can be used as an argument, why one can focus on	investigating bond polytopes of 3-connected graphs.
	
	Next, we turn our attention to edge- and cycle inequalities in bond polytopes, as they are known to be highly important in cut polytopes.
	In \Cref{sec:Non-Interleaved} we present \emph{non-interleaved} cycle inequalities, a class of facet-defining inequalities arising from a special class of cycles. After this, we discuss a generalization of such inequalities as well as edge inequalities in \Cref{sec:interleaved_cycles_and_edges}.
	
	We close this work by considering $(K_5-e)$-minor free graphs in \Cref{sec:K5-e_Minor}. We present a linear description of all bond polytopes of planar $3$-connected such graphs. 
	Combined with our reduction strategy from \Cref{sec:Reduction}, we can complement this with a
	linear-time algorithm for such graphs, improving (and fixing, see below) the current quadratic-time algorithm.

\mypar{A note on computing maximum bonds in $\boldsymbol{(K_5-e)}$-minor free graphs.}	

While the main focus of our work herein is to better understand the bond polytope and its facets, our results have direct algorithmic consequences---among others, on graphs with forbidden $(K_5-e)$-minor.
Recently, an algorithm was proposed to solve \MB on such graphs in quadratic time \cite{bond_K5-e_free}.
The key idea is to consider the graph's decomposition via 2-sums, and solving each component in quadratic time. However, the proposed algorithm's description is quite rough (e.g., it does not discuss how to efficiently obtain the 2-sum decomposition to start with) and contains a severe flaw, leading to an exponential instead of a quadratic overall running time: In~\cite{bond_K5-e_free} only the case of two subgraphs, joined via a 2-sum, is discussed (either their common vertices
are in the same partition side or not). One can hence compute both cases for both subgraphs and find the best choice. It is never discussed how to proceed in the case of more than two components. In fact, chaining this algorithm would yield an exponential running time of the order of $\Omega(2^c)$ for $c$ components; $(K_5-e)$-minor free graphs can have $c\in\Theta(n)$.

We resolve all these issues by giving an algorithm for the considered graph class that only requires linear running time.

\section{Preliminaries}\label{sec:Preliminaries}
In this section we provide some basic background on graphs and polytopes. Then, we recapitulate some known results on cut polytopes. For notation and results related to graphs we refer to \cite{diestel}, for those related to polytopes to \cite{BrunsGubeladze,ziegler}.

\mypar{Graphs.}
We only consider undirected graphs. A graph is \emph{simple}, if it does neither have parallel edges, nor self-loops.
Unless specified otherwise, we only consider simple connected graphs in the following. For $k \in \mathbb{N}$, let $[k]=\{1, \dots , k\}$.
Given a graph $G=(V,E)$ we also write $V(G)$ and $E(G)$ for its set of nodes $V$ and its set of edges $E$, respectively.
For $v,w \in V(G)$, we let $vw=\{v,w\}$ be the edge between $v$ and $w$.

A \emph{path} of length $k$ is a sequence of edges $e_1, \dots, e_k$ with $e_i =v_{i-1} v_i$ such that $v_i \neq v_j$ for $0 \leq i < j \leq k$.
Such a sequence but with $v_0 = v_k$ is a \emph{cycle} of length $k$; a cycle of length 3 is a \emph{triangle}.
A graph $H$ is a \emph{subgraph} of $G$, denoted by $H\subseteq G$, if (after possibly renaming) $V(H)\subseteq V(G)$ and $E(H) \subseteq E(G)$.
Given a subset $W \subseteq V$, the subgraph \emph{induced} by $W$ is the graph $G[W]=(W, \{uv \in E: u,v \in W \})$.
If an induced subgraph forms a cycle, it is an \emph{induced cycle} and as such \emph{chordless}.
We use $C_n$ to denote the cycle of length~$n$ and $K_n$ for the complete graph on $n$ nodes.
We denote the graph obtained from $G$ by deleting vertices $v_1,\dots v_k$ (resp. edges $e_1,\dots e_k$) by $G-\{v_1,\dots,v_k\}$ (resp. $G-\{e_1,\dots,e_k\}$). If we remove a single node $v$ (resp. edge $e$), we might just write $G-v$ (resp. $G-e$).
The graph $G/e$ is obtained from $G$ by \emph{contracting} the edge $e=vw$, i.e., the nodes $v$ and $w$ are identified, the arising self-loop is deleted and parallel edges are merged. $G$ contains an $H$-\emph{minor}, if $H$ can be obtained from $G$ by contraction and deletion of edges and vertices. Otherwise, $G$ is \emph{$H$-minor-free}.

$G$ is \emph{$k$-connected} if $|V(G)| \geq k+1$ and for each pair of nodes $v,w \in V(G)$ there exist $k$ internally node-disjoint paths from $v$ to $w$.
$1$-connected graphs, as well as the graph consisting of a single node and the empty graph are called \emph{connected}. 
If $G$ is connected but not $2$-connected, there exists some \emph{cut-node} $v \in V(G)$ such that $G-v$ is disconnected.

For two graphs $G$ and $H$, their \emph{union} $G \cup H=(V(G) \cup V(H), E(G) \cup E(H))$ is disjoint if their node sets are; in this case we may write $G \dotcup H$.
Assume two graphs $G$, $H$ contain $K_k$ as a subgraph, for some $k \in \mathbb{N}_{>0}$. The \emph{$k$-sum} (or \emph{clique-sum}) is obtained by taking the union of $G$ and $H$, identifying the $K_k$ subgraphs and possibly also removing edges contained in this specific $K_k$. A $k$-sum is \emph{strict}, if no edges are removed. We denote the strict $k$-sum of $G$ and $H$ by $G \oplus _k H$. Observe, that this notation does not explicitly state the specific $K_k$ in the question.	

\mypar{Bond- and Cut Polytopes.}
A \emph{polytope} $P$ is the convex hull of finitely many points in $\R^d$. 
The \emph{dimension} of $P$ is the dimension of its affine hull. We denote by $\cone(P)$ the \emph{conical hull} of $P$.
A linear inequality $a^\mathsf{T}x \leq b$ where $a \in \R^d$ and $b \in \R$ is a \emph{valid inequality} for $P$ if it is satisfied by all points $x \in P$.  
It is \emph{homogeneous} if $b=0$ and it is \emph{tight} if there is some $p \in P$ with $a^\mathsf{T}p=b$. We use the shorthand $\{a^\mathsf{T}x\leq b\}$ for $\{x \in \R^E : a^\mathsf{T}x \leq b\}$ and its analogon for equalities.
A (proper) \emph{face} of $P$ is a (non-empty) set of the form $P \cap \{a^\mathsf{T}x = b\}$ for some valid inequality $a^\mathsf{T} x \leq b$ with $a\neq \bnull$.
Faces of dimension~$0$ and $\dim(P)-1$ are \emph{vertices} and \emph{facets}, respectively. 
For each face $F \subseteq P$ there is a facet $F' \subseteq P$ dominating it, i.e., $F \subseteq F'$.
Two vertices are \emph{adjacent} if their convex hull is a $1$-dimensional face of $P$. 
Given a face $F \subseteq P$, its \emph{relative interior} $\relint(F)$ is the interior of $F$ with respect to its embedding into its affine hull.

A tight inequality $a^\mathsf{T}x \leq b$ is \emph{facet-defining} if $P\cap\{a^\mathsf{T}x = b\}$ is a facet of $P$.
Each polytope can be represented as the bounded intersection of finitely many closed half-spaces, i.e., $P$ admits a \emph{linear description} $P=\{x \in \R^d: Ax \leq {\bf b}\}$ for some matrix $A\in \R^{m \times d}$ and some vector ${\bf b}\in \R^m$.
This is given, e.g., by taking the system of all facet-defining inequalities.
\medskip

We now define the main objects of study in this paper.
Given a graph $G=(V,E)$ and a subset $S \subseteq V$, the set $\delta=\delta_G(S)=\delta_G(V\setminus S)=\{e \in E: |e \cap S| =1 \}$ is a \emph{cut} in $G$. We may omit the subscript when the graph is clear from the context.
If $G$ is connected, there are $2^{|V|-1}$ pairwise different cuts. The cut $\delta(S)$ is a \emph{bond} if the two graphs $G[S]$ and $G - S$ are connected.
To each cut $\delta$ in $G$ we associate its indicator vector $x^\delta \in \R^E$ given by
$$x^\delta_e= \begin{cases}	1, & \text{ if } e \in \delta;\\
0, & \text{ else.}			\end{cases}$$
The \emph{cut polytope}\footnote{Sometimes in the literature the cut polytope is denoted by $\CutP^\square$, while $\CutP$ denotes the cut cone~\cite{DL1,DL2,GeometryOfCutsAndMetrics}.} and the \emph{bond polytope} of $G$ are defined as
\begin{align*}
	\CutP(G)&=\conv(\{ x^\delta: \delta \text{ is a cut in } G \})\subseteq \R^E \text{ and} \\
	\Bond(G)&=\conv(\{ x^\delta: \delta \text{ is a bond in } G \})\subseteq \R^E \text{, respectively.}
\end{align*}
The cut polytope has dimension $\dim(\CutP(G))=|E(G)|$ \cite{bipartitesubgraph}.
Given a valid inequality $a^\mathsf{T}x \leq b$ of $\CutP(G)$ or $\Bond(G)$ its \emph{support graph} $\supp(a)\subseteq G$ is the subgraph of $G$ induced by the edge set $\{e \in E(G): a_e \neq 0\}$.

If $\supp(a)$ is a single edge, the inequality is called an \emph{edge inequality}. The \emph{homogeneous edge inequality} associated to an edge $e \in E(G)$ is $-x_e \leq 0$.

\section{First Properties and Comparison to $\boldsymbol{\CutP(G)}$}\label{sec:FirstProperties}
We start the study of bond polytopes by investigating their relation to cut polytopes. Afterwards we discuss whether some
fundamental results on cut polytopes carry over to bond polytopes.

We first observe that since by definition $\Bond(G) \subseteq \CutP(G)$ for any graph $G$, every facet-defining inequality of $\CutP(G)$ is valid for $\Bond(G)$.
In \cite{OnTheCutPolytope} it was shown that $1$-dimensional faces of $\CutP(G)$ can be characterized by bonds.

\begin{prop}\cite[Theorem 4.1]{OnTheCutPolytope}\label{prop:AdjacencyInCutP}
	Let $G=(V,E)$ be a connected graph and $\delta, \gamma \subseteq E$ be cuts. Then $x^\delta$ and $x^\gamma$ are the vertices of a $1$-dimensional face of $\CutP(G)$ if and only if their symmetric difference $\delta \triangle \gamma$ is a bond.
\end{prop}

As an almost immediate consequence we get an easy criterion for a vertex of $\CutP(G)$ being the incidence vector of a bond.

\begin{thm}\label{thm:Vertices}
	Let $G=(V,E)$ be a connected graph. Then, the following hold:
	\begin{enumerate}[(i)]
		\item The vertices of $\Bond(G)$ are $\bnull$ and its neighbors in $\CutP(G)$. In particular, $\cone(\CutP(G))=\cone(\Bond(G))$.
		\item $\dim\Bond(G) = |E|$.
		\item A homogeneous inequality $a^\mathsf{T}x \leq 0$ is facet-defining for $\Bond(G)$ if and only if it is facet-defining for $\CutP(G)$.
	\end{enumerate}
\end{thm}
\begin{proof}
	Statement {\it (i)} follows directly from \Cref{prop:AdjacencyInCutP}  and the fact that $\emptyset$ is a bond. Now, (ii) and (iii) are implied by {\it (i)}.
\end{proof}

Given a graph $G$, an edge $e \in E(G)$, and a set $S \subseteq V(G)$, the incidence vector of $\delta_{G-e}(S)$ is obtained from the incidence vector of $\delta_G(S)$ by removing the coordinate corresponding to $e$. As a consequence $\CutP(G-e)$ is the projection of $\CutP(G)$ onto the hyperplane $\{x_e=0\}$. 
The next example shows that this does not carry over to bond polytopes.
\begin{ex}
	For any $e=vw \in E(K_4)$ the cut $\delta_{K_4}(\{v,w\})$ is a bond in $K_4$.
	However, $\delta_{K_4-e}(\{v,w\})$ is a cut but no bond in $K_4-e$.
	In particular, $\Bond(K_4-e)$ is not the projection of $\Bond(K_4)$.
	~\hfill$\blacktriangleleft$
\end{ex}

Considering a graph $G$ and some facet-defining inequality $a^\mathsf{T}x \leq b$ of $\CutP(G)$ it is known that $a^\mathsf{T}x \leq b$ is a facet of $\CutP(\supp(a))$. This is not true in general for bond polytopes as shown in the following example.

\begin{ex}
	Considering $C_6 \subseteq K_{3,3}$, the inequality $\sum_{e \in E(C_{6})}x_e \leq 4$ is not even tight for $\Bond(C_6)$ but facet-defining for $\Bond(K_{3,3})$. Indeed, this example generalizes to $C_{2n}\subseteq V_n$ for arbitrary $n \geq 3$ (see \Cref{sec:interleaved_cycles_and_edges} for the definition of the Wagner graph~$V_n$).
	~\hfill$\blacktriangleleft$
\end{ex}

Conversely, it is well known that if $H \subseteq G$ is a subgraph, the 0-lifting, i.e., the lifting by taking the induced inequality in $\R^{E(G)}$, of each valid inequality of $\CutP(H)$ is valid for $\CutP(G)$. Again, this also does not carry over to bond polytopes.

\begin{ex}\label{ex:0lifting}
	Consider the graphs in \Cref{fig:Counterexample0lifting} and denote the outer (blue) cycle by~$C$. Then $\sum_{e \in E(C)} x_e \leq 2$ defines a facet of $\Bond(G)$, but this inequality is not even valid for $\Bond(G+e)$, as the (red) square nodes induce a bond $\delta$ with $|\delta \cap E(C)| >2$.~\hfill$\blacktriangleleft$
\end{ex}

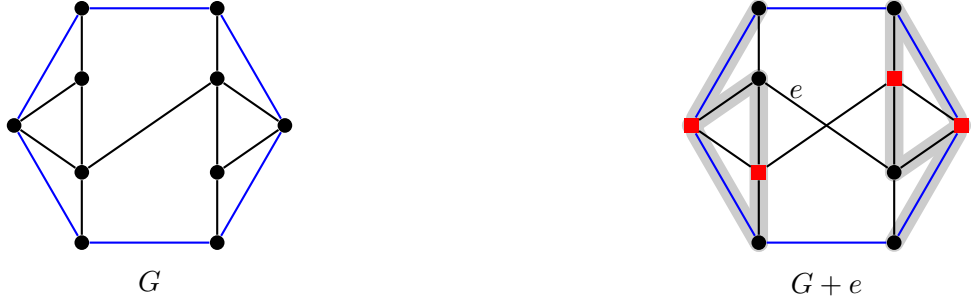
\begin{figure}
	\centering
	\begin{subfigure}[c]{.4\textwidth}
		\centering
			\begin{tikzpicture}
			\foreach \phi in {1,...,6}{
				\node[circle,fill, scale=0.5] (\phi) at (360/6 * \phi:1.8cm) {};
			}
			\draw[blue] (1)--(2)--(3)--(4)--(5)--(6)--(1);
			\node[circle,fill, scale=0.5] (a) at ($(4)!0.3!(2)$) {};
			\node[circle,fill, scale=0.5] (b) at ($(2)!0.3!(4)$) {};
			\node[circle,fill, scale=0.5] (c) at ($(5)!0.3!(1)$) {};
			\node[circle,fill, scale=0.5] (d) at ($(1)!0.3!(5)$) {};
			\draw (2) --(b)--(a)--(4);
			\draw (1) --(d)--(c)--(5);
			\draw (a)--(d);
			\draw (3) --(a);
			\draw (3) -- (b);
			\draw (6) -- (c);
			\draw (6) --(d);
			\end{tikzpicture}
		\caption*{$G$}
	\end{subfigure}
	\hfill
	\begin{subfigure}[c]{.4\textwidth}
		\centering
			\begin{tikzpicture}
			\node[circle,fill, scale=0.5] (1) at (360/6 * 1:1.8cm) {};
			\node[circle,fill, scale=0.5] (2) at (360/6 * 2:1.8cm) {};
			\node[red,rectangle,fill, scale=0.7] (3) at (360/6 * 3:1.8cm) {};
			\node[circle,fill, scale=0.5] (4) at (360/6 * 4:1.8cm) {};
			\node[circle,fill, scale=0.5] (5) at (360/6 * 5:1.8cm) {};
			\node[red,rectangle,fill, scale=0.7] (6) at (360/6 * 6:1.8cm) {};
			
			\draw[blue] (1)--(2)--(3)--(4)--(5)--(6)--(1);
			\node[red,rectangle,fill, scale=0.7] (a) at ($(4)!0.3!(2)$) {};
			\node[circle,fill, scale=0.5] (b) at ($(2)!0.3!(4)$) {};
			\node[circle,fill, scale=0.5] (c) at ($(5)!0.3!(1)$) {};
			\node[red,rectangle,fill, scale=0.7] (d) at ($(1)!0.3!(5)$) {};
			\draw (2) --(b)--(a)--(4);
			\draw (1) --(d)--(c)--(5);
			\draw (a)--(d);
			\draw (b)--node[xshift=-4mm, yshift=4.5mm]{\small $e$}(c); 
			\draw (3) --(a);
			\draw (3) -- (b);
			\draw (6) -- (c);
			\draw (6) --(d);
			\begin{pgfonlayer}{bg}
				\draw[gray!40,line width=7pt,line cap=round,rounded corners] (b.center)--(a.center)--(4.center);
				\draw[gray!40,line width=7pt,line cap=round,rounded corners] (1.center)--(d.center)--(c.center);
				\draw[gray!40,line width=7pt,line cap=round,rounded corners] (2.center)--(3.center)--(4.center);
				\draw[gray!40,line width=7pt,line cap=round,rounded corners] (5.center)--(6.center)--(1.center);
				\draw[gray!40,line width=7pt,line cap=round,rounded corners] (3.center)--(b.center);
				\draw[gray!40,line width=7pt,line cap=round,rounded corners] (6.center)--(c.center);
			\end{pgfonlayer}
			\end{tikzpicture}
		\caption*{$G+e$}
	\end{subfigure}
	\caption{Graphs from \Cref{ex:0lifting}. Marked edges in $G+e$ are those contained in the bond.}
	\label{fig:Counterexample0lifting}
\end{figure}

In contrast to this, contracting an edge $e$ corresponds to intersecting the bond polytope with the hyperplane $\{x_e=0\}$ as it is the case for cut polytopes.

\begin{obs}
	Let $G$ be a graph and $e\in E(G)$. Then, $\Bond(G/e)= \Bond(G) \cap \{x_e=0\}$.

\end{obs}

The most prominent symmetries of cut polytopes are given by graph automorphisms and \emph{switchings}.
\begin{lem}[Switching Lemma, see {\cite[Corollary 2.9.]{OnTheCutPolytope}}]
	Let $G=(V,E)$ be a graph and $a^\mathsf{T} x \leq b$ be a facet-defining inequality for $\CutP(G)$. Let $W \subseteq V$, and set 
	$b'= b - \sum_{e \in \delta(W)}a_e$ and 
	$a'_e=(-1)^{\mathds{1}[e \in \delta(W)]}\cdot a_e$ 
	for all $e\in E$.
	Then $(a')^\mathsf{T} x \leq b'$ defines a facet of $\CutP(G)$.
\end{lem}
While graph automorphisms clearly give rise to symmetries of bond polytopes, switching does not in general:

\begin{obs}
	Let $G=(V,E)$ be a graph, $W \subseteq V(G)$, and $a^\mathsf{T}x \leq b$ be facet-defining for $\Bond(G)$.
	We consider the following situations:
	\begin{enumerate}[(i)]
		\item If $b=0$ and $\delta(W)$ is a bond satisfying $a^\mathsf{T}x^{\delta(W)} =0$, switching $a^\mathsf{T}x \leq b$ at $W$ gives a facet of $\Bond(G)$.
		
		\item If $b=0$ and $\delta(W)$ is a bond with $a^\mathsf{T}x^\delta < 0$, the inequality obtained by switching $a^\mathsf{T}x\leq 0$ at $W$ is not facet-defining for $\Bond(G)$ in general.
		
		\item If $b \neq 0$, switching $a^\mathsf{T}x \leq b$ at a node set $W$ does not define a facet of $\Bond(G)$ in general. It might not even be valid for $\Bond(G)$ (even if $\delta(W)$ is a bond).
	\end{enumerate}
\end{obs}
\begin{proof}
	
	In statement \emph{(i)} we consider the switching of a homogeneous facet of $\CutP(G)$ at a solution of itself. This switching  yields a homogeneous facet of $\CutP(G)$ and thus a facet of $\Bond(G)$.
	
	Statement (ii) and (iii) can be shown via examples; we consider certain facets of $\Bond(C_n)$. The facet description of $\Bond(C_n)$ is discussed in \Cref{thm:FacetDescriptionCn}.
	For statement \emph{(ii)} consider the facet-defining inequality
	$$x_e - \sum_{\substack{f \in E(C_n)\\f \neq e}}x_f \leq 0.$$
	for some $e \in E(C_n)$.
	Switching this at $\{v\}$ for some $v \in V(C_n)$ that is not incident to $e$ we obtain the inequality
	$$x_e + \sum_{f \in \delta(\{v\})}x_f - \sum_{\substack{f \in E(C_n)\setminus \delta(\{v\})\\f \neq e}}x_f \leq 2.$$
	It follows directly from \Cref{thm:FacetDescriptionCn} that the latter does not define a facet.
	
	For statement \emph{(iii)} consider a cycle $C_n$ with $n \geq 4$ and the facet-defining inequality $\sum_{e \in E(C)}x_e \leq 2$ for $\Bond(C_n)$. Switching at an arbitrary node $v \in V(C_n)$ gives the inequality 
	$$\sum_{\substack{e \in E(C_n)\\v \notin e}}x_e-\sum_{\substack{e \in E(C_n)\\v \in e}}x_e \leq 0.$$
	But this is violated by $x^{\delta(\{w\})}$ for each $w \in V(C_n)$ that is not adjacent to $v$.
\end{proof}

\section{Constructing Facets from Facets}\label{sec:ConstructingFacetsFromFacets}
There are extensive studies considering the effect of graph operations (such as node splitting, edge subdivisions, edge contraction, and deletion of edges) on cut polytopes and their facet-defining inequalities \cite{OnTheCutPolytope}. Motivated by this, we start an investigation of the effect of such graph operations on bond polytopes and their facets.

\begin{thm}[Node splitting]\label{thm:nodesplitting}
	Let $G=(V,E)$ be a connected graph, $v \in V$, and $a^\mathsf{T}x \leq b$ be facet-defining for $\Bond(G)$.	
	Obtain $\widebar{G}=(\widebar{V},\widebar{E})$ as follows: replace $v$ by two adjacent nodes $v_1$ and $v_2$ and distribute the edges incident to $v$ arbitrarily among $v_1$ and $v_2$.
	Set $\varphi\colon\widebar{E}\setminus\{v_1v_2\} \to E$ by  
	$$\varphi(e)=\begin{cases}
		e, 	&\text{if } v_1,v_2\notin e, \\
		vw,	&\text{if } e=v_iw\  (i=1,2).
	\end{cases} $$
	Let $\omega$ be the value of a maximum bond in $\widebar{G} - v_1v_2$ separating $v_1$ and $v_2$ with respect to the edge weights given by $a_{\varphi(e)}$.
	Now, set $\widebar{a}$ by
	$$\widebar{a}_e=
	\begin{cases}
	a_{\varphi(e)}, &\text{if } e \neq v_1v_2,\\
	b-\omega,&\text{if $e=v_1v_2$.}\\
	\end{cases} $$
	Then $\widebar{a}^\mathsf{T} x \leq b$ defines a facet of $\Bond(\G)$.
\end{thm}

\begin{proof}
	First, we show that $\widebar{a}^\mathsf{T}x \leq b$ is a valid inequality for $\Bond(\widebar{G})$. 
	Each bond in $\G$ not containing $v_1v_2$ corresponds to a bond in $G$. Hence, it is easy to see that all such bonds satisfy the inequality under consideration.
	Now, let $\delta \subseteq \E$ be a bond with $v_1v_2 \in \delta$. Then 
	$$\widebar{a}^\mathsf{T}x^\delta =
	\a_{v_1v_2}x^\delta_{v_1v_2} + \sum_{\substack{e \in \E\\ e \neq v_1v_2}} \a_ex^\delta_e =
	(b-\omega)x^\delta_{v_1v_2} + \sum_{e \in E}a_ex^{\varphi(\delta)}_e  \leq
	 (b-\omega)+\omega =b.$$
	
	It remains to show that $\a^\mathsf{T}x \leq b$ is indeed facet-defining.
	Let $m=|E|$. Since $a^\mathsf{T}x \leq b$ defines a facet of $\Bond(G)$, there exist $W_1,\dots , W_m \subseteq V$ with $v \notin W_i$ such that  $x^{\delta_G(W_i)}$ satisfies $a^\mathsf{T}x =b$ for each $i \in [m]$ and $x^{\delta_G(W_1)}, \dots, x^{\delta_G(W_m)}$ are affinely independent.
	It is easy to see that $\delta_i=\delta_{\G}(W_i) \subseteq \widebar{E}$ is a bond (in $\widebar{G}$) with $v_1v_2 \notin \delta_i$ satisfying $\widebar{a}^\mathsf{T}x =b$.
	
	Now let $W_0 \subseteq \widebar{V}$ such that $\delta_{\G-{v_1v_2}}(W_0)$ is a bond in $\widebar{G} - v_1v_2$ separating $v_1$ and $v_2$ with $\sum_{e \in \delta_{\G-v_1v_2}(W_0)} a_e = \omega$. 
	Hence, for $\delta_0=\delta_{\G}(W_0) \subseteq  \widebar{E}$ we have
	$\widebar{a}^\mathsf{T} x^{\delta_0}= \omega +a_{v_1v_2}= \omega +b -\omega =b $.
	Since $v_1v_2 \in \delta_0$ and $v_1v_2 \notin \delta_i$ for $1 \leq i \leq m$,  it is easy to see that $x^{\delta_0},x^{\delta_1}, \dots, x^{\delta_m}$ are affinely independent.
\end{proof}

\begin{thm}[Replacing a node by a triangle]
	Let $G=(V,E)$ be a connected graph, $v \in V$, and $a^\mathsf{T}x \leq b$ be facet-defining for $\Bond(G)$.
	Obtain $\widebar{G}=(\V,\E)$ from $G$ by replacing $v$ by a triangle on vertices $\{v_1,v_2,v_3\}$ and distributing the edges incident to $v$ arbitrarily among $v_1$, $v_2$, and $v_3$.

	Set $\varphi\colon\E\setminus\{v_1v_2,v_1v_3,v_2v_3\} \to E$ by
	$$\varphi(e)=\begin{cases}
	e, 	&\text{if } v_1,v_2,v_3\notin e, \\
	vw,	&\text{if } e=v_iw\ (i\in[3]).
	\end{cases}$$
	For $i =1,2,3$, let $S_i \subseteq \V$ such that $S_i \cap \{v_1,v_2,v_3\}=\{v_i\}$ and $\delta_\G(S_i)$ is a maximum bond with respect to edge weight $0$ attached to $v_jv_k$ ($j,k\in [3]$) and $a_{\varphi(e)}$ for each other edge $e$. Denote by $\omega_i$ the weight of $\delta_\G(S_i)$ with respect to these weights.
	Define $\a \in \R^{\E}$ by
	$$\a_e=\begin{cases}
				a_{\varphi(e)}, &\text{if } e \neq v_iv_j\, (i,j \in [3]),\\
				\frac{1}{2}(b-\omega_1-\omega_2+\omega_3), &\text{if } e=v_1v_2,\\
				\frac{1}{2}(b-\omega_1+\omega_2-\omega_3), &\text{if } e=v_1v_3,\\
				\frac{1}{2}(b+\omega_1-\omega_2-\omega_3), &\text{if } e=v_2v_3.
			\end{cases}$$
			Then $\a^\mathsf{T}x \leq b$ defines a facet of $\Bond(\G)$.
\end{thm}

\begin{proof}
	The proof of this theorem is analogous to the one of \Cref{thm:nodesplitting}. We only have to utilize that $x^{\delta_G(S_1)},x^{\delta_G(S_2)},x^{\delta_G(S_3)}$ all satisfy $\a^Tx=b$ and are affinely independent to all bond-vectors $x^{\delta_\G(S)}$ with $S \subseteq V$.
\end{proof}

\begin{lem}\label{lem:SolExistance}
	Let $G=(V,E)$ be a connected graph, $e \in E$, and $a^\mathsf{T}x \leq b$ be facet-defining for $\Bond(G)$. 
	Assume that $\{a^\mathsf{T}x \leq b \}\neq \{-x_e\leq 0\}$ and  $a_e \neq 0$.
	Then there exists some bond $\delta$ such that $e \in \delta$ and $a^\mathsf{T}x^\delta=b$.
\end{lem}
\begin{proof}
	Assume there is no such $\delta$ and consider the inequality $a^\mathsf{T}x- \lambda x_e \leq b$ for some $\lambda>0$. Validity follows since $a^\mathsf{T}x- \lambda x_e  \leq a^\mathsf{T}x$ for each $x \in [0,1]^{|E|}$. Moreover, by assumption each bond satisfying  $a^\mathsf{T}x = b$ satisfies $a^\mathsf{T}x- \lambda x_e  = b$. Hence, both inequalities describe the same face of $\Bond(G)$ contradicting $a^\mathsf{T}x \leq b$ being facet-defining.
\end{proof}

\begin{thm}[Subdividing an edge]
	Let $G=(V,E)$ be a connected graph and $a^\mathsf{T}x \leq b$ be facet-defining for $\Bond(G)$. Let $e=vw \in E$ with $a_e \leq \frac{b}{2}$.
	
	Obtain $\widebar{G}=(\widebar{V},\widebar{E})$ by splitting $e$ into $e_1=v \widebar{v}$ and $e_2=\widebar{v} w$ (for a new node $\ov$). For $f \in \widebar{E}$ set
	$$\widebar{a}_f=
	\begin{cases}
	a_f, &\text{if } f \in E,\\
	a_e, &\text{if } f \in\{e_1,e_2\}.
	\end{cases} $$
	Then $\widebar{a}^\mathsf{T}x \leq b$ is facet-defining for $\Bond(\widebar{G})$.
\end{thm}

\begin{proof}
	For each cut $\delta \subseteq \widebar{E}$ in $\G$ we have $|\delta \cap \{e_1,e_2\}|\leq 1$ or $\delta=\delta_{\G}(\widebar{v})$. Hence, it is easy to see that $\widebar{a}^\mathsf{T}x \leq b$ is valid for $\Bond(\widebar{G})$.
	
	Since $a^\mathsf{T}x \leq b$ is facet-defining for $\Bond(G)$ there exist $U_1, \dots, U_m \subseteq V$ ($m=|E|$) such that each $\delta_G (U_i)$ is a bond in $G$ satisfying $a^\mathsf{T}x^{\delta(U_i)}=b$.
	We may assume $v \in U_i$ for all $i\in [m]$. Setting $\widebar{U}_i=U_i \cup \{\widebar{v}\}$ it is easy to see that $\delta_i=\delta_{\G}(\widebar{U}_i)$ is a bond satisfying $\widebar{a}^\mathsf{T} x^{\delta_i} =b$. Note that $e_1 \notin \delta_i$ for each $i \in [m]$.
	By \Cref{lem:SolExistance} there exists $U_0 \subseteq V$ such that $v \in U_0$, $w \notin U_0$, and $a^\mathsf{T}x^{\delta_G(U_0)}=b$.
	We conclude that $e_1 \in \delta_{\G}(U_0)$ and thus, $x^{\delta_\G(U_0)},x^{\delta_1}, \dots, x^{\delta_m}$ are affinely independent. Hence, $\widebar{a}^\mathsf{T}x \leq b$ is facet-defining for $\Bond(\widebar{G})$.
\end{proof}

Iteratively applying this theorem yields the following corollary:
\begin{cor}[Replacing an edge by a path]
	Let $G=(V,E)$ be a connected graph and $a^\mathsf{T}x \leq b$ be facet-defining for $\Bond(G)$. Let $e=vw \in E$ and assume that $a_e \leq \frac{b}{2}$.
	
	Obtain $\widebar{G}=(\widebar{V},\widebar{E})$ by subdividing $e$ into edges $e_1,\dots , e_k$ for arbitrary $k\in \mathbb N$, $k \geq 2$. For $f \in \widebar{E}$ set
	$$\widebar{a}_f=
	\begin{cases}
	a_f, &\text{if } f \in E,\\
	a_e, &\text{if } f \in\{e_1,\dots,e_k\}.
	\end{cases} $$
	Then $\widebar{a}^\mathsf{T} x \leq b$ is facet-defining for $\Bond(\widebar{G})$.
\end{cor}

Next, we consider the inverse, i.e., the replacement of an induced path by an edge. To this end, we start by investigating coefficients of facet-defining inequalities on edges contained in induced paths.

\begin{lem}\label{lem:path_coeffs}
	Let $G=(V,E)$ be a connected graph, $P \subseteq G$ be an induced path, $a^\mathsf{T}x \leq b $ be facet-defining for $\Bond(G)$, and $\calE$ be the set of all bonds $\delta$ in $G$ satisfying the equality $a^\mathsf{T}x^\delta = b$.
	Assume that $E(P) \cap \supp(a) \neq \emptyset$ and $\{a^\mathsf{T}x \leq b \} \neq \{-x_e \leq 0 \}$ for each $e \in E(P)$.
	Then:
	\begin{enumerate}[(i)]
		\item If $|\delta \cap E(P)| \leq 1$ for all $\delta\in \calE$, we have  $a_e=a_f$ for all $e,f \in E(P)$.
		\item If there exists some $\delta^* \in \calE$ with $|\delta^* \cap E(P)|=2$,
		 we have either $a_e=a_f$ for all $e,f \in E(P)$ or there is a unique $\e \in E(P)$ such that $a_\e>a_e$ and $a_e=a_f$ for all $e,f \in E(P)\setminus \{\e\}$.
	\end{enumerate}
	In particular, we always have $a_e\neq 0$ for some $e \in E(G)\setminus E(P)$ or $\max_{e \in E(P)}a_e >0$.
\end{lem}

\begin{proof}
	By \cite[Theorem 5]{lectures_0-1-polytopes} we may assume $a \in \Z^E$ and since $\bnull \in \Bond(G)$ we have $b \in \Z_{\geq 0}$.
	Let $M= \max_{e \in E(P)}\{a_e\}$ and $\e \in E(P)$ with $a_\e=M$. Set $N=\max_{e \in E(P)\setminus \{\e\}}\{a_e\}$ and define $c \in \R ^E$ by
	$$c_e=\begin{cases}
		a_e,	&\text{if $e \notin E(P)$},\\
		M,		&\text{if $e =\e$},\\
		N,		&\text{else}.
	\end{cases}$$
	Note that we might have $M=N$.
	Let $\delta$ be an arbitrary bond in $G$.
	If $\delta \cap E(P) = \emptyset$, we have $a^\mathsf{T}x^\delta=c^\mathsf{T}x^\delta\leq b$.
	If $|\delta \cap E(P)|=\{f\}$ for some $f \in E(P)$, we consider the bond $\delta'=(\delta \setminus \{f\})\cup \{\e\}$ and obtain $c^\mathsf{T}x^\delta = a^\mathsf{T}x^{\delta'} \leq b$.
	If $|\delta \cap E(P)|=2$, we consider the bond $\delta'= (\delta \setminus (\delta \cap E(P))) \cup \{\e, f^*\}$ for some $f^* \in E(P)$ with $a_{f^*}=N$ and obtain $c^\mathsf{T}x^\delta =a^\mathsf{T}x^{\delta'} \leq b$.
	Hence, $c^\mathsf{T}x \leq b$ is valid for $\Bond(G)$.
	Since $a^\mathsf{T} x \leq b $ can be obtained from $c^\mathsf{T}x \leq b$ by adding homogeneous edge inequalities, we have either $c=a$ or $c=\bnull$. The latter case implies that $a^\mathsf{T}x\leq b$ is a homogeneous edge inequality contradicting the assumption.
	 
	Thus, in the following we can assume that $a=c$. Then, statement \emph{(ii)} follows from the previous discussion.
	To prove statement \emph{(i)} assume, that $| \delta \cap E(P)|\leq 1$ for all $\delta \in \calE$ and assume that $m<M$. Defining $d \in \R^E$ by
	$$ d_e= \begin{cases}
	a_e,	&\text{if $e \notin E(P)$},\\
	M,		&\text{if $e =\e$},\\
	m+1,	&\text{else},\end{cases}$$
	we show that $a^\mathsf{T}x \leq b$ cannot be facet-defining for $\Bond(G)$.
	Similar as above, we have $d^\mathsf{T}x^\delta \leq b$ for all bonds $\delta$ with $|\delta \cap E(P)| \leq 1$. 
	
	Since $P$ is an induced path, a bond in $G$ contains at most $2$ edges from $P$.
	Hence, it remains to prove validity for all bonds picking $2$ edges from $P$. To this end, let $\delta$ be such a bond. By assumption we have $a^\mathsf{T}x^\delta <b$ and thus, since both values are integers it follows that $b-1 \geq a^\mathsf{T}x^\delta =M+N$. Hence, $d^\mathsf{T}x^\delta =M+N+1 \leq b$.
	Since $d=\bnull$ would imply that $a^\mathsf{T}x \leq b$ is a homogeneous edge inequality, $d^\mathsf{T}x \leq b$ defines a proper face of $\Bond(G)$.
	Thus, $d^\mathsf{T}x \leq b$ dominates $a^\mathsf{T}x \leq b$ contradicting the assumption that the latter inequality is facet-defining. 
	
	The \qq{in particular}-part holds since otherwise $a^\mathsf{T}x\leq b$ is dominated by homogeneous edge inequalities.
\end{proof}

\begin{ex}
	Indeed, all cases of the previous lemma can occur. 
	We may see this considering facet-defining inequalities of $\Bond(C_n)$, which are formally discussed later in \Cref{thm:FacetDescriptionCn}.
	Let $e \in E(C_n)$, and consider the facet-defining inequalities
	\begin{align}
		\sum_{e  \in E(C_n)}x_e &\leq 2 \label{eq:noninterleaved_cycle}\quad\text{ and}\\
		x_e - \sum_{\substack{f \in E(C_n)\\f \neq e}} x_f &\leq 0 \label{eq:homogeneous_cycle}.
	\end{align}
	An example for statement \emph{(i)} is given by inequality~(\ref{eq:homogeneous_cycle}) and an arbitrary induced path $P$ not containing $e$.
	For the first case of statement \emph{(ii)} consider inequality~(\ref{eq:noninterleaved_cycle}) and an arbitrary induced path $P\subseteq C_n$; for the second case consider inequality~(\ref{eq:homogeneous_cycle}) and an induced path $P \subseteq C_n$ with $e\in E(P)$.
\end{ex}

Given the previous lemma, we are now prepared to investigate, how replacing a path by an edge effects facet-defining inequalities of bond polytopes.

\begin{thm}[Replacing a path by an edge]
	Let $G=(V,E)$ be a connected graph, $a^\mathsf{T}x \leq b$ be facet-defining for $\Bond(G)$, and $P$ be an induced path in $G$.
	Denote by $\calE$ the set of all bonds $\delta$ in $G$ satisfying the equality $a^\mathsf{T}x^\delta = b$.
	Assume that $\{a^\mathsf{T}x \leq b \} \neq \{-x_e \leq 0 \}$ for each $e \in E(P)$ and either there exist $e,f \in E(P)$ with $a_e \neq a_f$ or there exists no bond $\delta\in \calE$ with $|\delta \cap E(P)|=2$.
	
	Let $\widebar{G}=(\widebar{V},\widebar{E})$ be obtained from $G$ by replacing $P$ by a single edge $p$. Set $M=\max_{e \in E(P)}\{a_e\}$ and define $\a \in \R^{\widebar{E}}$ by
	$$\a_e=
	\begin{cases}
	a_e, &\text{if $e \in E$},\\
	M, &\text{if e=p.}
	\end{cases} $$
	Then $\widebar{a}^\mathsf{T}x \leq b$ defines a facet of $\Bond(G)$.
\end{thm}
\begin{proof}
	Let $\pi\colon \R^E \to \R^\E$ be the projection given by
	$$\pi(x)_e=\begin{cases}
	x_e, 					&\text{if }e \neq p,\\
	\sum_{e \in E(P)}x_e, 	&\text{if }e=p.
	\end{cases} $$
	If there is no $\delta^* \in \calE$ with $|\delta^*\cap E(P)|=2$, it is straight forward to see that for each bond~$\delta \in \calE$ we have $\pi(x^\delta) \in \Bond(\G)$ and $\a^\mathsf{T}\pi(x^\delta)=b$. Thus, 
	\begin{align*}
		\dim \left(\{\a^\mathsf{T}x =b\}\cap \Bond(\G)\} \right)&\geq \dim \left(\{a^\mathsf{T}x =b\}\cap \Bond(G) \right) -\left( |E(P)|-1 \right)\\
		&=|E|-1-\left( |E(P)|-1 \right)=|\E|-1.
	\end{align*}
	Since $a^\mathsf{T}x \leq b$ was facet-defining, we have $\a\neq \bnull$. Moreover, since $\Bond(\G)$ has dimension $|\E|$, we have $\{\a^\mathsf{T}x =b\}\cap \Bond(\G)\neq \Bond(\G)$ yielding that $\a^\mathsf{T}x \leq b$ is facet-defining.
	
	Now, assume there exists a bond $\delta^*\in \calE$ with $|\delta^* \cap E(P)|=2$ and edges $e,f \in E(P)$ with $a_e \neq a_f$. By \Cref{lem:path_coeffs}, there exists a unique edge  $\e \in E(P)$ with $a_\e > a_f$ for all $f \in E(P)\setminus\{e^*\}$ and we have $a_e=a_f$  for all $e, f \in E(P)\setminus\{e^*\}$.
	Note that for each bond $\delta \in \calE$ with $\delta \cap E(P) \neq \emptyset$ we have $\e \in \delta$. Thus, there are exactly 
	$|E(P)|-1$ such bonds with $|\delta \cap E(P)|=2$. Since $\delta \cap E(P)=\{\e\}$ for all bonds with $|\delta \cap E(P)| = 1$, applying $\pi$ on all bonds $\delta$ with $a^\mathsf{T}x^\delta=b$ and $|\delta \cap E(P)| \leq 1$ yields a $\left(|\E|-1\right)$-dimensional face of $\Bond(\G)$. Thus, $\a^\mathsf{T}x \leq b$ is facet-defining for $\Bond(\G)$.
\end{proof}

We close this section by discussing a bond polytope version of \cite[Lemma~2.5]{OnTheCutPolytope}:
\begin{lem}\label{lem:BarahonaZeroKoeff}
	Let $G=(V,E)$ be a connected graph and $a^\mathsf{T}x \leq b$ be a valid inequality for $\Bond(G)$. Moreover, let $pq \in E$ and let $S\subsetneq V \setminus \{p,q\}$.
	
	If $\delta(S)$, $\delta(S \cup \{p\})$, $\delta(S \cup \{q\})$, and $\delta(S \cup \{p,q\})$ are bonds satisfying  $a^\mathsf{T}x = b$, then $a_{pq}=0$.
\end{lem}
\begin{proof}
	One can re-use the proof of \cite[Lemma 2.5]{OnTheCutPolytope} by simply replacing \qq{cut} with \qq{bond}.
\end{proof}
The cut version \cite[Lemma 2.5]{OnTheCutPolytope} of the above lemma turns out to be a powerful tool when applied to cut versions of the other above lemmata and theorems in this section:
It allows the addition of further edges to the graph while retaining the facet-defining properties of the inequalities under consideration.
In the context of bond polytopes, however, there often may not exist a set $S$ yielding the required bonds (in contrast to cuts).
Nonetheless, although not as versatile as its cut version, the lemma will still be crucial in later proofs.

\section{Reduction to $\boldsymbol{3}$-connectivity}\label{sec:Reduction}

We show that $\MB$ can be reduced in linear time to $3$-connected graphs.
While a (similar) reduction was already proposed in \cite{bond_K5-e_free}, it contained a gap, leading to an exponential running time. 
Both our reduction and the one in \cite{bond_K5-e_free} is based on the following observation. For completeness we give a full proof.

\begin{obs}[Bonds over clique sums]\label{lem:BondsOverCliqueSums}
	Let $G=G_1 \oplus_k G_2$ with $k\in [2]$ and $\delta_G \subseteq E(G)$ be a bond. Let $\delta_i=\delta_G \cap E(G_i)$ for $i \in [2]$.
	\begin{enumerate}[(i)]
		\item If $k=1$, then $\delta_i=\delta_G$ and $\delta_{3-i}=\emptyset$ for either $i=1$ or $2$.
		\item If $k=2$, let $e$ be the unique edge in $E(G_1)\cap E(G_2)$. Either $e \in \delta(G)$ and each $\delta_i$ is a bond in $G_i$ or $e \notin \delta_G$ and $\delta_G=\delta_i$ for either $i=1$ or $2$.
	\end{enumerate}
\end{obs}
\begin{proof}
	We prove statement (ii). It is then straight forward to verify statement (i).
	Let $V(G_1) \cap V(G_2)=\{v_1,v_2\}$, $e=v_1v_2$, and $S \subseteq V(G)$ such that $\delta_{G}=\delta_G(S)$.
	Setting $S_i=S \cap V(G_i)$ and $\delta_i=\delta_{G_i}(S_i)$ for $i=1,2$ we have  $\delta_{G}= \delta_1 \cup \delta_2$. Clearly, $\delta_1$ (resp. $\delta_2$) is a bond in $G_1$ (resp. $G_2$) and we have $v_1v_2 \in \delta_G$ if and only if $v_1v_2 \in \delta_1$ and $v_1v_2 \in \delta_2$.
	It only remains to show that $v_1v_2 \notin \delta_G$ implies $\delta_1=\emptyset$ or $\delta_2=\emptyset$.	
	If $v_1v_2 \notin \delta$, we may assume $v_1,v_2 \in S$.
	Since $G-v_1v_2=(G_1-v_1v_2 ) \dotcup (G_2-v_1v_2)$ there cannot be $w_1,w_2 \in V(G)\setminus S$ with $w_i \in V(G_i)$ because $G-S$ would be disconnected. Thus, $V(G_i) \subseteq S$ and $\delta_{G}$ is a bond in $G_{3-i}$ for either $i=1$ or $2$.
\end{proof}

In order to use this observation algorithmically, we first have to discuss how to efficiently find (and work with) the necessary decomposition of an arbitrary graph into its 2-sum components.

Let $G=(V,E)$ be a $2$-connected multigraph (i.e., we allow parallel edges) and let the nodes $\{v,w\}$ be a \emph{split pair} in $G$, i.e., $G-\{v,w\}$ is disconnected or there are parallel edges connecting $v$ and $w$.
The \emph{split classes} of $\{v,w\}$ are given by a partition $E_1, \dots ,E_k$ of $E$ such that two edges are in a common split class if and only if there is a path between them neither containing $v$ nor $w$ as an internal node. As $G$ is $2$-connected, it is easy to see that $v$ and $w$ are both incident to every split class. 
For a split class $C$ let $\widebar{C}=E\setminus C$.
A \emph{Tutte split} replaces $G$ by the two multigraphs $G_1=(V(C),C \cup \{e\})$ and $G_2=(V(\widebar{C}), \widebar{C} \cup \{e\})$, provided that $G_1-e$ or $G_2-e$ remains $2$-connected.
Thereby, $e$ is a new \emph{virtual} edge connecting $v$ and $w$.
Observe that this operation may yield parallel edges.
Iteratively splitting multigraphs via Tutte splits gives the unique \emph{$3$-connectivity decomposition} 
of $G$. Its components, the so-called \emph{skeletons}, can be partitioned into the following sets: a set $S$ of cycles, a set $P$ of edge bundles (two nodes joined by at least $3$ edges), and a set $R$ of simple $3$-connected graphs. See, e.g., \cite{tutte,triconnected}.

We use a data structure to efficiently consider the components of the $3$-connectivity decomposition of~$G$.
The \emph{SPR-tree} $T=T(G)$ has a node for each element in $S$, $P$, and~$R$~\cite{dibatt,spr}\footnote{The data structure is also known as \emph{SPQR-tree}. However, the originally proposed nodes of type $Q$ (as well as the tree's orientation) have often turned out to be superfluous.}. 
For a node $\alpha \in V(T)$, let $H_\alpha$ denote its corresponding skeleton. 
Two nodes $\alpha,\beta \in V(T)$ are adjacent if and only if $H_\alpha$ and $H_\beta$ share a virtual edge. 
$G$ can be reconstructed from $T$ by taking the non-strict $2$-sum of components whenever their corresponding nodes are adjacent in $T$. 
Following this interpretation, $P$-nodes containing a non-virtual edge represent strict $2$-sums of their adjacent components of the decomposition. 
$T$ has only linear size and can be computed in $\mathcal{O}(|E|)$ time~\cite{triconnected}.

\begin{thm}\label{thm:Reduction_3connectivity_Components}
	\MB on arbitrary graphs can be solved in the same time complexity as on (simple) $3$-connected graphs.
\end{thm}

\begin{proof}
	
	Let $G$ be an arbitrary graph and $c_e$ denote the edge weight of $e \in E(G)$.
	We denote by $\omega(G)$ the weight of a maximum bond in $G$.
	Moreover, let $p(m)\in \Omega(m)$ be the running time of \MB on $3$-connected graphs with $m$ edges.
	
	If $G$ is not $2$-connected, we can find a decomposition $G=G_1 \oplus_1 \dots \oplus_1 G_k$ into $2$-connected graphs $G_1, \dots, G_k$ in linear time with simple depth-first search.
	We can consider these components individually since
	$\omega(G)=\max_{i\in [k]}\omega(G_i)$ (see \Cref{lem:BondsOverCliqueSums}).
	Thus we may assume $G$ to be $2$-connected in the following.
	
	For an edge $e \in E(G)$, we can compute a maximum bond not containing $e$ (resp.\ containing $e$) 
	in the same time as $\omega(G)$ by contracting $e$ (resp.\ setting its weight to a large enough value, e.g., $2\sum_{f\in E(G)\setminus \{e\}}|c_f|$).
	
	First, we compute the SPR-tree $T=T(G)$ of $G$, attach weight $0$ to each virtual edge and root $T$ at an arbitrary node $\varrho \in V(T)$.
	Our algorithm will iteratively prune leaves in $T$.
	
	Let $\alpha$ be a leaf of $T$, $\beta$ be its parent, and denote by $e$ the common virtual edge in $H_\alpha$ and $H_\beta$. We compute the value $\omega^+_\alpha$ of a maximum bond in $H_\alpha$ containing $e$
	and the value $\omega^-_\alpha$ of a maximum bond not containing $e$.
	Then, we set the weight of $e$ in $H_\beta$ to $\omega^+_\alpha$, consider $e$ as an original (no longer virtual) edge in $H_\beta$, remove $\alpha$ from $T$, and proceed with the next leaf until only $\varrho$ remains. In the latter case, we compute $\omega_\varrho$ as the maximum bond in the skeleton $H_\varrho$ (where all virtual edges are already transformed into original edges with some weight computed in the previous steps).
	
	Consider the above setting when considering a leaf $\alpha$,
	and let $\delta^*$ be a maximum bond in $G$. 
	In case of $\delta^* \subseteq E(H_\alpha)\setminus\{e\}$, we have $\omega(G)=\omega^-_\alpha$.
	Otherwise, let $G'$ be the graph obtained from $T$ after the removal of $\alpha$ (in particular, $e$ is considered an original edge in $G'$ of weight $\omega^+_\alpha$).	
	Then, we have:
	\begin{align*}
		\omega(G)=\sum_{f \in \delta^*}c_f&=\sum_{\substack{f \in \delta^*\cap E(G')\\f \neq e}}c_f+\sum_{\substack{f \in \delta^*\cap E(H_\alpha)\\f \neq e }}c_f \\
		&=\sum_{\substack{f \in \delta^*\cap E(G')\\f \neq e}}c_f+
		\begin{cases}
		\omega^+_\alpha, &\text{if $e\in \delta^*$}
		\\
		0, &\text{if $e \notin \delta^*$}
		\end{cases}
		\\
		&=\omega(G').
	\end{align*}
	Overall, we have $\omega(G)=\max\{\omega^-_\alpha,\omega(G')\}.$
	Iterating our pruning strategy, we obtain 
	$$\omega(G)=\max\left\{\omega_\varrho, \max_{\alpha \in V(T(G))\setminus \varrho}\{\omega^-_\alpha \}\right\}.$$
	Observe that $T(G)$ is the original SPR-tree and the $\omega$-values are the maximum bonds as computed by the algorithm.
		
	Note that computing $\MB$ on $P$- and $S$-nodes can trivially be done in linear time: a maximum bond in $P$-nodes either picks all edges or no edge; a maximum bond in $S$-nodes picks either two edges of heaviest weight or none, if the sum of any two edge weights is negative.
	The SPR-tree can be built in linear time and the computations in a skeleton on $m'$ edges can be done in time $\mathcal{O}(p(m'))$.
	We attain an overall running time of
	$\mathcal{O}\Big(\sum_{\alpha\in V(T)}p(|E(H_\alpha)|) \Big) \leq \mathcal{O}(p(|E(G)|))$.
\end{proof}

Although not necessary in the above proof, we may mention that the bond polytope corresponding to a $P$-node is essentially that of a single edge; the bond polytope corresponding to an $S$-node, i.e., of a simple cycle, is discussed in~\Cref{thm:FacetDescriptionCn} below.

\section{Non-Interleaved Cycle Inequalities}\label{sec:Non-Interleaved}

We will now start the investigation of inequalities associated to cycles.
To this end, we introduce non-interleaved cycles and show that they give rise to facet-defining inequalities for bond polytopes of $3$-connected graphs.
We close this section by discussing this inequalities for $2$-connected graphs.

\begin{defi}
Let $G$ be a graph and $C \subseteq G$ be a cycle. $C$ is \emph{interleaved}, if there are (not necessarily neighboring) nodes $v_1, v_2,v_3, v_4 \in V(C)$ occurring along $C$ in this order such that there are node-disjoint paths in $G -E(C)$ connecting $v_1$ with $v_3$ and $v_2$ with $v_4$ respectively. Otherwise, $C$ is \emph{non-interleaved}.
\end{defi}

Given an interleaved cycle $C$ we call two paths $P$, $Q$ witnessing the interleavedness of $C$ \emph{interleaving} (with respect to $C$).
 Two such paths can be found in polynomial time if they exist \cite{GraphMinors13}.
The following lemma introduces valid inequalities for bond polytopes, which we call $\emph{non-interleaved cycle inequalities.}$

\begin{lem}\label{lem: InterleavedCycle_Valid}
	Let $G$  be a connected graph and $C \subseteq G$ be a cycle.
	The inequality $\sum_{e \in E(C)}x_e \leq 2$ is valid for $\Bond(G)$ if and only if $C$ is non-interleaved.
\end{lem}
\begin{proof}
	First assume $C$ is interleaved with interleaving paths $P$ and $Q$. 
	Set $$S=V(P)\cup\{v \in V(G): v \text{ is disconnected from $Q$ in $G\setminus P$} \}.$$
	Then $\delta(S)$ is a bond with $|\delta(S) \cap E(C)|=4$. Hence, $\sum_{e \in E(C)}x_e \leq 2$ is not valid.
	
	On the other hand, if $\sum_{e \in E(C)}x_e \leq 2$ is violated there exists
	a bond $\delta(S')$ with $|\delta(S')\cap C|>2$.
	Since there always has to be an even number of cut edges on a cycle, we must have $|\delta(S')\cap C|\geq 4$.
	Let $P_1, \dots, P_\ell$ be the components of $C - \delta(S')$ listed in their order of appearance along $C$. 
	We may assume $P_i \subseteq G[S']$ for odd~$i$.
	Then there needs to exist a path $Q_1\subseteq G[S']$ connecting $P_1$ and $P_3$ and a path $Q_2\subseteq G-S'$ connecting $P_2$ and $P_4$. 
	Clearly, $Q_1$ and $Q_2$ are interleaving with respect to $C$.
\end{proof}

\begin{lem}\label{lem:InterleavedCycle_Indced}
	Let $G$ be $3$-connected and $C\subseteq G$ be a non-interleaved cycle. Then $C$ is an induced cycle.
\end{lem}
\begin{proof}
	Let $G$ be $3$-connected and assume there is a non-interleaved cycle $C$ with a chord $e$. Let $C_1$ and $C_2$ be the cycles such that $E(C_1) \cap E(C_2)=\{e\}$ and $C=C_1 \triangle C_2$.
	There is at least one node in $C_1 \setminus C_2$ and at least one node in $C_2 \setminus C_1$. Furthermore, since $G$ is $3$-connected there exists a path $P$ with $P\cap E(C)=\emptyset$ connecting these two nodes. But then $P$ and $e$ are interleaving paths with respect to $C$.
\end{proof}

The following graph-theoretic lemma is the crucial ingredient for the facet theorem shown thereafter.

\begin{lem}\label{lem:constuct_set_S}
	Let $G=(V,E)$ be $3$-connected, $C \subseteq G$ be a cycle, and $f=pq \in E \setminus E(C)$. Then there exists an $S \subseteq V$ such that $\delta(S)$, $\delta(S \cup \{p\})$,  $\delta(S \cup \{q\})$, and $\delta(S \cup \{p,q\})$ are bonds each containing two edges of $C$.
\end{lem}

\begin{proof}
	We consider two cases depending on whether $pq$ is adjacent to $C$ or not.

	{\it Case 1:} $|\{p,q\} \cap V(C)| =1$. 
		We may assume $p \in V(C)$ and $q \notin V(C)$. Since $G$ is $3$-connected, there exist internally node-disjoint paths $P$ and $Q$ with $f \notin P,~Q$ connecting $p$ and $q$. We may assume $|E(P) \cap E(C)| \geq 1$, since $p$ is not a cut-node. 
		Now, set
		$$S^+=V(P) \cup \{v \in V:v \text{ is disconnected from $C \setminus P$ in $G\setminus P$}\}$$
		and $S=S^+ \setminus \{p,q\}$.
		In the following, we show that the cuts $\delta(S)$, $\delta(S \cup \{p\})$, $\delta(S \cup \{q\})$, and $\delta(S^+)$ are indeed bonds.
		To this end it suffices to prove that $G[S]$ and $G - S^+$ are both connected, and $p$ and $q$ are adjacent to both of these graphs.
		
		By construction, $G- S^+$ is connected, and the nodes $p$ and $q$ are adjacent to $G[S]$. Moreover, $p$ is incident to two edges in $C$ and only one of these is contained in $G[S^+]$. Thus, $p$ is adjacent to $G- S^+$. 
		Furthermore, the path $Q$ connects $q$ and $C$, yielding that $q$ is adjacent to $G - S^+$.
		For each node $v \in S$ there are three disjoint paths $P'_1$, $P'_2$, and $P'_3$ connecting $v$ and $C$.
		Since $v \in S$ we have $V(P'_i) \cap V(P) \neq \emptyset$ for each $i\in [3]$.
		Since at most two of these paths contain $p$ or $q$, there is a path connecting $v$ with some node in $V(P)\setminus \{p,q\}$. 
		Hence, $G[S]$ is connected.
		
	{\it Case 2:} $p,q \notin V(C)$.
		We prove that there are paths $P_1, P_2$ connecting~$p$ with~$C$, and $Q_1,Q_2$ connecting~$q$ with $C$ such that the paths of the triplets $(P_1,P_2,Q_1)$ and $(P_1,Q_1,Q_2)$ are pairwise disjoint within their triplet.
		Then, we may assume (possibly after exchanging indices) that there is a path $R \subseteq C$ connecting $P_1$ and~$Q_1$ such that $R \cap P_2= \emptyset$ and $R \cap Q_2=\emptyset$. Then, setting 
		$S'=V(R) \cup V(P_1) \cup V(Q_1)$ it is straight forward to verify that
		$$S^+=S' \cup \{v \in V(G):v \text{ is disconnected from $C \setminus R$ in $G-S'$}\}$$
		and $S=S^+\setminus \{p,q\}$ yield the claimed bonds (cf. \Cref{fig:Construct_S+}).
		
		We start by proving the existence of a triplet $(P_1,P_2,Q)$ of disjoint paths $P_1,P_2$ connecting $p$ with $C$ and $Q$ connecting $q$ with $C$. 
		For a path $P$ and nodes $v,w \in V(P)$ we denote by $P[v{:}w]$ the subpath of $P$ from $v$ to $w$.
		Since $G$ is $3$-connected, by Menger's theorem, there exist three disjoint (except of $p$) paths $P_1,P_2,P_3$ connecting~$p$ with~$C$.
		If one of them, say $P_3$, contains $q$ we are done by choosing $Q=P_3 \setminus (P_3[p{:}q]-q)$.
		 Otherwise, since $G$ is $3$-connected, there exists a path $Q_0$ connecting $q$ and $C$.
		If $Q_0$ is not the claimed path $Q$, it intersects at least one of the $P_i$.
		We may assume that the first such intersection is between $Q_0$ and $P_3$ and is at a node $x$.
		Then, we set $Q=Q_0[q{:}x] \cup P_3[x{:}c]$.
		By construction, $Q$ is disjoint from $P_1$ and $P_2$.

		Now, we construct $Q_1$ and $Q_2$ given the disjoint triplet of paths $(P_1,P_2,Q)$ such that $P_1=(p,p_1^1\dots,c_1)$, $P_2=(p,p_1^2\dots,c_2)$, and $Q=(q,q_1\dots ,c)$. We may assume that $c_1,c_2,c \in V(C)$ are pairwise distinct (cf. \Cref{fig:ConstructS_a}).
		Since $G$ is $3$-connected, there exists a path $Q'$ connecting $q$ and $C$ such that $p,q_1 \notin Q'$.
		If $Q'$ is disjoint from $P_1,P_2$ and $Q$ we can set $Q_1=Q$ and $Q_2=Q'$.
		So, assume $Q'$ is not disjoint from $P_1$, $P_2$, $Q$.
		Let $x$ denote the first node along $Q'$ (starting at $q$) such that $x \in P_1$, $x \in P_2$, or $x \in Q$ (cf. \Cref{fig:ConstructS_b}).
		If $x \in P_i$, $i \in [2]$, we set $Q_1=Q$ and $Q_2=Q'[q{:}x] \cup P_i[x{:}c_2]$.
		Now, assume $x \in Q$.
		Since $G$ is $3$-connected, there exists a path $Q''$ connecting a node $a \in V(Q[q{:}x])\setminus\{q,x\}$ with $C$ such that the two neighbors of $a$ in $Q$ are not contained in $Q''$.
		If $Q''$ is disjoint from $P_1,P_2,Q$, we set $Q_1=Q[q{:}a] \cup Q''$ and $Q_2=Q'[q{:}x] \cup Q[x{:}c]$.
		So, assume $Q''$ intersects $P_1$, $P_2$ or $Q$ and denote by $y$ the first such intersection along $Q''$ starting in $a$ (cf. \Cref{fig:ConstructS_c}). 
		We may assume $p \notin V(Q'')$ because if there was no such path, $G -\{x,p\}$ would be disconnected.
		Moreover, we may assume $y \notin V(Q)$ because otherwise $G - \{y,p\}$ would be disconnected.
		Thus  (after possibly renaming), we may assume $y \in P_2[p_1^2{:}c_2]$.
		We have the claimed paths $Q_1=Q'[q{:}x] \cup Q[x{:}c]$ and $Q_2= Q[q{:}a] \cup Q''[a{:}y] \cup P_1[y{:}c_1]$.
\end{proof}

\begin{figure}
		\begin{subfigure}[c]{.5\textwidth}
		\centering
		\begin{tikzpicture}
			\node (label) at (-1.3,-.3) {$C$};
			\draw[black] ([shift=(10 :1)]0,0) arc (10:170:1cm);
			\draw[red] ([shift=(10 :1)]0,0) arc (10:-190:1cm);
			\node [circle,fill,scale=.5] (d1) at (40 :1) {} ;
			\node [red,circle,fill,scale=.5] (d2) at (10 :1) {} ;
			\node [circle,fill,scale=.5] (c1) at (140 :1) {} ;
			\node [red,circle,fill,scale=.5] (c2) at (170 :1) {} ;		
			
			\node[red] (P) at (-90 :.7) {\footnotesize$R$} ;	
			
			\node[draw,circle,fill,scale=.5] (p) at (-1,3) {};
			\node[draw,circle,fill,scale=.5] (q) at ( 1,3) {};
			\node (p') at ($(p)+(-.2,.2)$){\footnotesize$p$};
			\node (q') at ($(q)+(.2,.2)$){\footnotesize$q$};
			
			\draw (p)--(q);
			\draw (p)--node[xshift=3mm]{\footnotesize$P_2$} (c1);
			\draw[red] (p) edge[in=130,out=-110] node[xshift=-3mm , yshift=0mm]{\footnotesize$P_1$} (c2);
			\draw (q)--node[xshift=-3mm]{\footnotesize$Q_2$} (d1);
			\draw[red] (q) edge[in=50,out=-70] node[xshift=3mm , yshift=0mm]{\footnotesize$Q_1$} (d2);
			
		\end{tikzpicture}
	\end{subfigure}%
	\hfill
	\begin{subfigure}[c]{.5\textwidth}
		\centering
		\begin{tikzpicture}
		
		\node (label) at (-1.3,-.3) {$C$};
		\draw[black] ([shift=(10 :1)]0,0) arc (10:170:1cm);
		\draw[red] ([shift=(10 :1)]0,0) arc (10:-190:1cm);
		\node [red,circle,fill,scale=.5] (d2) at (10 :1) {} ;
		\node [red,circle,fill,scale=.5] (c2) at (170 :1) {} ;		
		
		\node[red] (P) at (-90 :.7) {\footnotesize$R$} ;	
		
		\node[draw,circle,fill,scale=.5] (p) at (-1,3) {};
		\node[draw,circle,fill,scale=.5] (q) at ( 1,3) {};
		\node (p') at ($(p)+(-.2,.2)$){\footnotesize$p$};
		\node (q') at ($(q)+(.2,.2)$){\footnotesize$q$};
		\node [circle,fill,scale=.5] (c) at (90 :1) {} ;
		\node [circle,fill,scale=.5] (d) at (90 :2) {} ;
		
		\draw (p)--(q);
		\draw[red] (p) edge[in=130,out=-110] node[xshift=-3mm , yshift=0mm]{\footnotesize$P_1$} (c2);
		\draw[red] (q) edge[in=50,out=-70] node[xshift=3mm , yshift=0mm]{\footnotesize$Q_1$} (d2);
		\draw (p)--(d);
		\draw (q)--(d);
		\draw (d) -- node[xshift=-4mm , yshift=4mm] {\footnotesize$P_2$}node[xshift=4mm,yshift=4mm]{\footnotesize$Q_2$}(c);
		\end{tikzpicture}
	\end{subfigure}%
	
	\caption{The set $S$ (red) in the second case of the proof of \Cref{lem:constuct_set_S}}\label{fig:Construct_S+}
\end{figure}
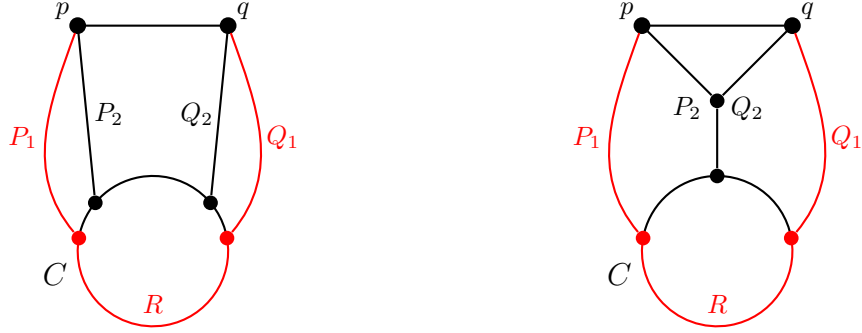

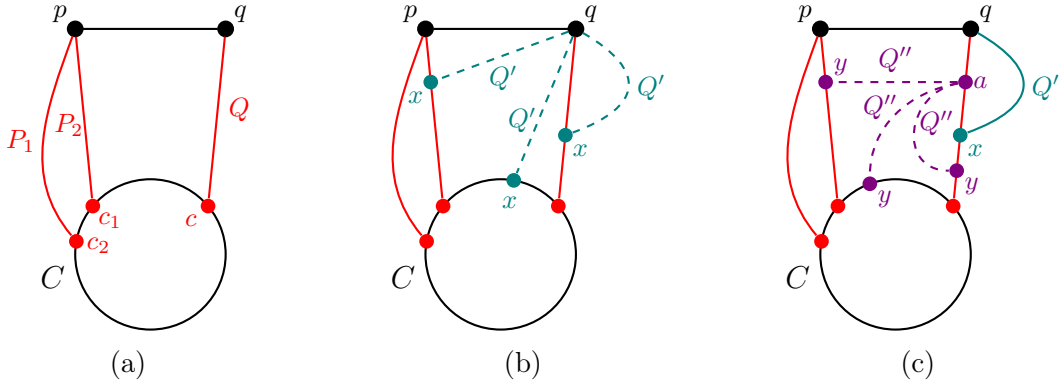
\begin{figure}
	\begin{subfigure}[c]{.3\textwidth}
		\centering
		\begin{tikzpicture}
			\node (label) at (-1.3,-.3) {$C$};
			\draw (0,0) circle[radius=1cm];
			\node [red, circle,fill,scale=.5] (c) at (40 :1) {} ;
			\node [red, circle,fill,scale=.5] (c1) at (140 :1) {} ;
			\node [red, circle,fill,scale=.5] (c2) at (170 :1) {} ;
			\node [red] (c')  at ( 40 :.7) {\footnotesize$c$} ;
			\node [red] (c1') at (140 :.7) {\footnotesize$c_1$} ;
			\node [red] (c2') at (170 :.7) {\footnotesize$c_2$};

			\node[draw,circle,fill,scale=.5] (p) at (-1,3) {};
			\node[draw,circle,fill,scale=.5] (q) at ( 1,3) {};
			\node (p') at ($(p)+(-.2,.2)$){\footnotesize$p$};
			\node (q') at ($(q)+(.2,.2)$){\footnotesize$q$};
			
			\draw (p)--(q);
			\draw[red] (p)--node[xshift=-2mm , yshift=-1mm]{\footnotesize$P_2$} (c1);
			\draw[red] (p) edge[in=130,out=-110] node[xshift=-3mm , yshift=0mm]{\footnotesize$P_1$} (c2);
			\draw[red] (q)--node[xshift=3mm , yshift=1mm]{\footnotesize$Q$} (c);
		\end{tikzpicture}
		\subcaption{}\label{fig:ConstructS_a}
	\end{subfigure}%
	\hfill
	\begin{subfigure}[c]{.3\textwidth}
		\centering
		\begin{tikzpicture}
			\node (label) at (-1.3,-.3) {$C$};
			\draw (0,0) circle[radius=1cm];
			\node [red, circle,fill,scale=.5] (c) at (40 :1) {} ;
			\node [red, circle,fill,scale=.5] (c1) at (140 :1) {} ;
			\node [red, circle,fill,scale=.5] (c2) at (170 :1) {} ;
			\node[draw,circle,fill,scale=.5] (p) at (-1,3) {};
			\node[draw,circle,fill,scale=.5] (q) at ( 1,3) {};
			\node (p') at ($(p)+(-.2,.2)$){\footnotesize$p$};
			\node (q') at ($(q)+(.2,.2)$){\footnotesize$q$};
			
			\draw (p)--(q);
			\draw[red] (p)-- (c1);
			\draw[red] (p) edge[in=130,out=-110] (c2);
			\draw[red] (q)-- (c);
			
			\node[circle,fill,scale=.5, blue!50!green] (x1) at (80:1){};
			\node[blue!50!green] (x1') at (80:.7) {\footnotesize$x$};
			\node[circle,fill,scale=.5, blue!50!green] (x2) at ($(p)!0.3!(c1)$){};
			\node[blue!50!green] (x2') at ($(x2)+(-.2,-.2)$){\footnotesize$x$};
			\node[circle,fill,scale=.5, blue!50!green] (x3) at ($(q)!0.6!(c)$){};
			\node[blue!50!green] (x3') at ($(x3)+(.2,-.2)$){\footnotesize$x$};
			\draw[dashed,blue!50!green] (q)--node[xshift=-3mm , yshift=-2mm]{\footnotesize$Q'$}(x1);
			\draw[dashed,blue!50!green] (q)--node[xshift=0mm , yshift=-3mm]{\footnotesize$Q'$}(x2);
			\draw[dashed,blue!50!green] (q)edge[in=20,out=-30,distance=1cm]node[xshift=3mm , yshift=0mm]{\footnotesize$Q'$}(x3);
		\end{tikzpicture}
		\subcaption{}\label{fig:ConstructS_b}
	\end{subfigure}%
	\hfill
	\begin{subfigure}[c]{.3\textwidth}
		\centering
		\begin{tikzpicture}
			\node (label) at (-1.3,-.3) {$C$};
			\draw (0,0) circle[radius=1cm];
			\node [red, circle,fill,scale=.5] (c) at (40 :1) {} ;
			\node [red, circle,fill,scale=.5] (c1) at (140 :1) {} ;
			\node [red, circle,fill,scale=.5] (c2) at (170 :1) {} ;
			\node[draw,circle,fill,scale=.5] (p) at (-1,3) {};
			\node[draw,circle,fill,scale=.5] (q) at ( 1,3) {};
			\node (p') at ($(p)+(-.2,.2)$){\footnotesize$p$};
			\node (q') at ($(q)+(.2,.2)$){\footnotesize$q$};
			\draw (p)--(q);
			\draw[red] (p)--(c1);
			\draw[red] (p) edge[in=130,out=-110](c2);
			\draw[red] (q)--(c);
			
			\node[circle,fill,scale=.5, blue!50!green] (x3) at ($(q)!0.6!(c)$){};
			\draw[blue!50!green] (q)edge[in=20,out=-30,distance=1cm]node[xshift=3mm , yshift=0mm]{\footnotesize$Q'$}(x3);
			\node[blue!50!green] (x3') at ($(x3)+(.2,-.2)$){\footnotesize$x$};
			
			\node[circle,fill,scale=.5, blue!50!red] (a) at ($(q)!0.3!(c)$){};
			\node[blue!50!red] (a') at ($(a)+(.2,0)$){\footnotesize$a$};
			\node[circle,fill,scale=.5, blue!50!red] (y1) at (110:1){};
			\node[blue!50!red] (y1') at ($(y1)+(.2,-.2)$){\footnotesize$y$};
			\node[circle,fill,scale=.5, blue!50!red] (y2) at ($(p)!0.3!(c1)$){};
			\node[blue!50!red] (y2') at ($(y2)+(.2,.2)$){\footnotesize$y$};
			\node[circle,fill,scale=.5, blue!50!red] (y3) at ($(q)!0.8!(c)$){};
			\node[blue!50!red] (y3') at ($(y3)+(.2,-.2)$){\footnotesize$y$};
			\draw[dashed, blue!50!red] (a)edge[in=90,out=-170]node[xshift=-2mm , yshift=1mm]{\footnotesize$Q''$}(y1);
			\draw[dashed, blue!50!red] (a)--node[xshift=0mm , yshift=3mm]{\footnotesize$Q''$}(y2);
			\draw[dashed, blue!50!red] (a)edge[in=180,out=-170,distance=.7cm]node[xshift=3mm , yshift=1mm]{\footnotesize$Q''$}(y3);
		\end{tikzpicture}
		\subcaption{}\label{fig:ConstructS_c}
	\end{subfigure}%
	\caption{Visualization of obtaining the necessary paths in the second case of the proof of \Cref{lem:constuct_set_S}}
\end{figure}

\begin{thm}\label{thm:NonInterleavedCycleFacet}
	Let $G=(V,E)$ be $3$-connected and $C \subseteq G$ be a non-interleaved cycle. Then $\sum_{e \in E(C)}x_e \leq 2$ defines a facet of $\Bond(G)$.
\end{thm}

\begin{proof}
	Since by \Cref{lem: InterleavedCycle_Valid}, $\sum_{e \in E(C)}x_e \leq 2$ is valid for $\Bond(G)$, there is a facet-defining inequality $a^\mathsf{T}x \leq b$ such that 
	$\Bond(G)\cap\{\sum_{e \in E(C)}x_e = 2 \} \subseteq \Bond(G)\cap\{a^\mathsf{T}x = b \}$.
	We show equality of these two faces by proving that $a_e=a_f$ for all $e,f \in E(C)$ and $a_f=0$ for each $f \notin E(C)$.

	Let $f=pq \in E(G) \setminus E(C)$. By \Cref{lem:constuct_set_S} there is a set $S \subseteq V$ such that $\delta(S)$, $\delta(S \cup \{p\})$, $\delta(S \cup \{q\})$, and $\delta(S \cup \{p,q\})$ are bonds, each satisfying $\sum_{e \in E(C)}x_e = 2$ and thus, $a^\mathsf{T}x=b$. Hence, \Cref{lem:BarahonaZeroKoeff} yields $a_{pq}=0$.
	Note that for each $v \in V(C)$ and $vw \in E(C)$ the bond vectors  $x^{\delta(\{v\})}$ and $x^{\delta(\{v,w\})}$ satisfy $\sum_{e \in E(C)}x_e = 2$ and thus, $a^\mathsf{T}x=b$. Considering a $3$-path on nodes $v_1,\dots, v_4$ labeled in order of their appearance along $C$, we have
	\begin{align*}
		0 &= a^\mathsf{T}x^{\delta(v_2)}-a^\mathsf{T}x^{\delta(v_3)}=(a_{v_1v_2}+a_{v_2v_3}) -(a_{v_2v_3}-a_{v_3v_4})=a_{v_1v_2}-a_{v_3v_4},\\
		0 &= a^\mathsf{T}x^{\delta(\{v_2,v_3\})}-a^\mathsf{T}x^{\delta(v_3)}=(a_{v_1v_2}+a_{v_3v_4}) -(a_{v_2v_3}-a_{v_3v_4})=a_{v_1v_2}-a_{v_2v_3}.
	\end{align*}
	Thus, we have $a_{v_1v_2}=a_{v_2v_3}=a_{v_3v_4}$ yielding $a_e=a_f$ for all $e,f \in E(C)$.
\end{proof}

We close this section by discussing non-interleaved cycle inequalities in $2$-connected graphs.
Even though validity of the non-interleaved cycle inequalities is maintained when only assuming $2$-connectivity, such inequalities are not facet-defining in general for bond polytopes of non-$3$-connected graphs.

\begin{ex}\label{obs:2co0nnectedPlanar}
	Let $G$ be the graph shown in \Cref{fig:4gonIn4gon}. Computing the facet description of $\Bond(G)$, e.g., using the software package \emph{Normaliz} \cite{normaliz}, we see the following:
	\begin{itemize}
		\item Consider the cycle induced by $v_1,v_2,v_3,v_4$. The non-interleaved cycle inequality associated to this cycle is not facet-defining for $\Bond(G)$.
		\item Consider the cycle induced by $v_1,v_2,v_3,v_7,v_6,v_5$. The non-interleaved cycle inequality associated to this cycle is not facet-defining for $\Bond(G)$.
		~\hfill$\blacktriangleleft$
	\end{itemize}

\end{ex}

\begin{figure}
	\centering
	\begin{tikzpicture}[scale=.7]
	\node[circle,draw,scale=.8] (a) at (0,0) {$v_1$};
	\node[circle,draw,scale=.8] (b) at (4,0) {$v_2$};
	\node[circle,draw,scale=.8] (c) at (4,4) {$v_3$};
	\node[circle,draw,scale=.8] (d) at (0,4) {$v_4$};
	\node[circle,draw,scale=.8] (1) at (1,1) {$v_5$};
	\node[circle,draw,scale=.8] (2) at (3,1) {$v_6$};
	\node[circle,draw,scale=.8] (3) at (3,3) {$v_7$};
	\node[circle,draw,scale=.8] (4) at (1,3) {$v_8$};
	\draw (1)--(2)--(3)--(4)--(1);
	\draw (a)--(b)--(c)--(d)--(a);
	\draw (1)--(a);
	\draw (3)--(c);
	\end{tikzpicture}
	\caption{Graph from \Cref{obs:2co0nnectedPlanar}}\label{fig:4gonIn4gon}
\end{figure}

However, we can give a necessary condition for a non-interleaved cycle in a $2$-connected graph giving rise to a facet-defining inequality. To do this, we call a non-interleaved cycle $C$ \emph{maximal non-interleaved} if there is no non-interleaved cycle $C'$ with $|E(C) \setminus E(C')|=1$.

\begin{thm}
	Let $G$ be a connected graph and $C \subseteq E(G)$ be a cycle. If $\sum_{e \in E(C)}x_e \leq 2$ defines a facet of $\Bond(G)$, $C$ is maximal non-interleaved.
\end{thm}
\begin{proof}
	Assume $C$ is not maximal. Let $C'$ be non-interleaved with $C \setminus C'=\{f\}$. Then, $\sum_{e \in E(C)}x_e \leq 2$ is the sum of the inequalities $\sum_{e \in E(C)'}x_e \leq 2$ and $x_f - \sum_{e \in E(C)'\setminus C}x_e \leq 0$. As shown in \Cref{lem: InterleavedCycle_Valid}, the first of the two inequalities is valid for $\Bond(G)$ and by \cite[Theorem 3.3]{OnTheCutPolytope} the latter is valid for $\CutP(G)$ and thus for $\Bond(G)$. Hence, $\sum_{e \in E(C)}x_e \leq 2$ cannot be a facet of $\Bond(G)$.
\end{proof}

There is a class of simple non-3-connected graphs such that the non-interleaved cycle inequalities are not only facet-defining but together with the homogeneous (cut polytope) facets suffice to fully describe their bond polytopes: 

\begin{thm}\label{thm:FacetDescriptionCn}
	For each $n \geq 3$, $\Bond(C_n)$ is completely defined by the following facet-defining inequalities 
	\begin{align*}
	-x_e &\leq 0 											\qquad \text{for each $e\in E(C_n)$,}\\
	x_e - \sum_{f \in E(C_n) \setminus\{e\}} x_f &\leq 0	\qquad \text{for each $e \in E(C_n)$,}\\
	\sum_{e \in E(C_n)}x_e &\leq 2.
	\end{align*}
\end{thm}
\begin{proof}	
	By \cite[Corollary 3.10]{OnTheCutPolytope}, the homogeneous inequalities above are the homogeneous facets of $\CutP(C_n)$. Thus, by \Cref{thm:Vertices} these are exactly the homogeneous facets of $\Bond(C_n)$.
	Let $\delta \subseteq E(C_n)$ be a cut. Then $\delta$ is a bond if and only if $\delta=\emptyset$ or $|\delta| =2$. Thus, the inequality $\sum_{e\in E(C_n)}x_e\leq 2$ defines a facet of $\Bond(C_n)$.
	
	Now, let $Q$ denote the polytope defined by the above inequalities. Clearly, we have $\Bond(C_n) \subseteq Q$.
	Assume that $Q$ has a vertex $v$ not resembling a bond. Since all cuts $\delta$ in $C_n$ satisfying $|\delta|=2$ are bonds, $v$ is given as $\{v\}=\relint(F) \cap \{\sum_{e \in E(C_n)}x_e=2\}$ for some $1$-dimensional face $F$ of $\CutP(C_n)$. Since all incidence vectors of nonempty bonds are contained in $\{\sum_{e \in E(C_n)}x_e=2\}$, we have $x^\emptyset \in F$. By \Cref{prop:AdjacencyInCutP}, the second vertex of $F$ is the incidence vector of some bond, contradicting the existence of~$v$.
\end{proof}

\section{Edge- and Interleaved Cycle Inequalities}\label{sec:interleaved_cycles_and_edges}
Finally, we discuss edge-inequalities and a natural generalization of non-interleaved cycle inequalities.
To tackle the latter, we consider the intersection of bonds and interleaved cycles in a given graph.

\begin{lem}\label{obs:GeneralizedCycleInequality}
	Let $G$ be a graph, $C\subseteq G$ be a cycle and $k \in \mathbb{N}$. Then $\sum_{e \in E(C)}x_e \leq 2k$ is valid for $\Bond(G)$ if and only if $G$ does not contain a minor $H$ of the following form:
	$H=T \dotcup T'$ where $T$ and $T'$ are disjoint trees, each on $k+1$ nodes that correspond to nodes in~$C$ alternating around~$C$.
	If $k$ is chosen minimally, the inequality is tight.
\end{lem}

\begin{proof}
	Let $\mathcal T \subseteq G$ be a subgraph whose contraction gives $T\subseteq H$. By adding to $\mathcal T$ any components not connected to $C$ in $G \setminus \mathcal T$, $\delta=\delta_G(\mathcal T)$ becomes a bond with $|\delta \cap E(C)|=2(k+1)$.
	Conversely, if there is a bond $\delta=\delta(S) \subseteq E(G)$ with $|\delta \cap E(C)|= \ell > 2k$, $\ell$ even, this gives $\ell \geq 2(k+1)$ components in $C \setminus \delta$. Since $\delta$ is a bond, both $G[S]$ and $G-S$ contain trees as minors whose nodes correspond to these components.
	
	Now, let $k$ be minimal such that  $\sum_{e \in E(C)}x_e \leq 2k$ is valid for $\Bond(G)$. Then there exists some bond $\delta$ in $G$ such that 
	$2(k-1) <|\delta \cap E(C)|\leq 2k$. Tightness of the inequality follows since the number of cut edges in a cycle is always even. 
\end{proof} 

Indeed, such inequalities are facet-defining for some graphs. One class of such graphs are generalized \emph{Wagner graphs} $V_n$ ($n\in 2\mathbb N$) also known as circulants $C_{n}(1,\frac{n}{2})$: $V_n$ is obtained from the cycle $C_{n}$ on nodes $[n]$ by adding the edges $\{i,i+\frac{n}{2}\}$ for $1 \leq i \leq \frac{n}{2}$.
We call $C_{n}$ the \emph{outer cycle} of $V_n$.

\begin{thm}\label{thm:InterleavedCycleFacetVn}
	Let $n \geq 6$ and $C$ be the outer cycle of $V_n$. Then, the inequality $\sum_{e \in E(C)} x_e \leq 4$ defines a facet of $\Bond(V_n)$.
\end{thm}
\begin{proof}
	By \Cref{obs:GeneralizedCycleInequality}, the inequality  $\sum_{e \in E(C)} x_e \leq 4$ is valid for $\Bond(V_n)$. Thus, there is a facet-defining inequality $a^\mathsf{T}x \leq b$ of $\Bond(V_n)$ dominating it.
	We show $\{a^\mathsf{T}x=b\}=\{\sum_{e \in E(C)}x_e=4\}$ by proving $a_f =0$ for each $f \notin E(C)$ and $a_e=a_f$ for each $e,f \in E(C)$.

	First, we show that $a_{pq}=0$ for each $pq \notin E(C)$. 
	For this, let $v$ be a neighbor of $p$ and $w$ be a neighbor of $q$ such that $vw \in E(V_n)\setminus E(C)$ and set $S=\{v,w\}$.
	Then, $\delta(S)$, $\delta(S\cup \{p\})$, $\delta(S\cup \{q\})$, and $\delta(S\cup \{p,q\})$ are bonds satisfying $\sum_{e \in E(C)} x_e = 4$ and thus, $a^\mathsf{T}x =b$. Hence, \Cref{lem:BarahonaZeroKoeff} yields $a_{pq}=0$.
	
	It remains to show that $a_e=a_f$ for all $e,f \in E(C)$. It suffices to prove this for two incident edges $e,f \in E(C)$. Let $\{w\} = e \cap f$, $e=vw$ and $u \in V(V_n)$ the unique node with $uv \in E(V_n) \setminus E(C)$.
	Then $\delta(\{u,v\})$ and $\delta(\{u,v,w\})$ are bonds satisfying $\sum_{e \in E(C)} x_e = 4$ and thus, $a^\mathsf{T}x =b$.
	Since only edges in $C$ have non-zero coefficients, it follows that
	$0=a^\mathsf{T}x^{\delta(\{u,v\})}-a^\mathsf{T}x^{\delta(\{u,v,w\})}=a_f-a_e$.
\end{proof}

On the other hand:
\begin{ex}
	Consider $K_5$ and a $5$-cycle $C \subseteq K_5$. The inequality $\sum_{e \in E(C)} x_e \leq 4$ is valid and tight but not facet-defining.
\end{ex}

\begin{restatable}{question}{questone}
	Characterize interleaved cycles that induce facets. 
\end{restatable}

We close this section by  discussing inequalities associated to edges. By definition, the inequality $x_e \leq 1$ is always valid for $\Bond(G)$.
In the following, we show that although this inequality is not facet-defining in general, there is an infinite class of graphs where it is.

\begin{lem}
	Let $G=(V,E)$  be a connected graph and $e \in E$.
	If $e$ is contained in a non-interleaved cycle, $x_e \leq 1$ is not facet-defining for $\Bond(G)$.
\end{lem}
\begin{proof}
	Let $e$ be an edge contained in a non-interleaved cycle $C$. By \Cref{lem: InterleavedCycle_Valid} and \cite[Theorem 3.3]{OnTheCutPolytope}, the inequalities
	$$ \sum_{f \in E(C)}x_f \leq 2 \qquad \text{and} \qquad  x_e -\sum_{\substack{f \in E(C)\\f \neq e}} x_f \leq 0$$
	are valid for $\Bond(G)$. Summing these two inequalities, we obtain $2x_e \leq 2$.
\end{proof}

\begin{thm}\label{thm:EdgeIneqWagner}
	For any $n\geq 6$ and any $e\in E(V_n)$ the inequality $x_e \leq 1$ is facet-defining for $\Bond(V_n)$.
\end{thm}
\begin{proof}
	We use the same strategy as in the proof of \Cref{thm:InterleavedCycleFacetVn}.
	Let $e\in E(V_n)$.
	Since $x_e \leq 1$ is valid for $\Bond(V_n)$ there exists a facet-defining inequality $a^\mathsf{T}x \leq b$ dominating it. We show $\{a^\mathsf{T}x =b\}=\{x_e=1\}$ by proving $a_f=0$ for each $f \in E(V_n)\setminus \{e\}$.
	
	First assume $f \cap e =\emptyset$. Labeling the vertices along the outer cycle by $[n]$, up to isomorphism it suffices to consider the following four cases (cf. \Cref{fig:ProofEdgeWagnerSetS}):
	If $e=\{1,n\}$ and $f=\{i,i+1\}$ for $2 \leq i \leq \frac{n}{2}$, we set $S^+=[\frac{n}{2}+1]$;
	if $e=\{1,n\}$ and $f=\{i,\frac{n}{2}+i\}$ for $2 \leq i \leq \frac{n}{2}-1$, we set $S^+=[i] \cup ([\frac n 2 +i]\setminus [\frac{n}{2}])$;
	if $e=\{\frac{n}{2},n\}$ and $f=\{i,i+1\}$ for $1 \leq i \leq \frac{n}{2}-2$ we set $S^+=[i+1] \cup ([n]\setminus [\frac{n}{2}+i])$;
	if $e=\{\frac{n}{2},n\}$ and $f=\{i,\frac{n}{2}+i\}$ for $1 \leq i \leq \frac{n}{2}-1$ we set $S^+=[i] \cup ([n]\setminus [\frac{n}{2}+i-1])$. It is straight forward to verify that for each of these sets, $S^+$, $S^+\setminus\{i\}$, $S^+\setminus\{j\}$, and $S^+\setminus\{i,j\}$ (where $j$ is the other
	end node of $f$) induce bonds satisfying $x_e=1$ and thus $a^\mathsf{T}x =b$. Hence, \Cref{lem:BarahonaZeroKoeff} yields $a_f=0$.
	
	It remains to show that each edge incident to $e$ has coefficient $0$.
	Depending on whether $e$ is contained in the outer cycle or not, we are in one of the situations sketched in \Cref{fig:ProofEdgeWagner}.
	In both cases, considering the notation as in the figure,
	all bond-vectors in the inequalities below satisfy the equalities $x_e=1$ and thus $a^\mathsf{T}x=b$. Since $a_h=0$ for each edge $h \in E(V_n)\setminus\{e,f_1,f_2,f_3,f_4\}$, we have
	\begin{alignat*}{4}
		b 	&=a^\mathsf{T}x^{\delta(\{v\})}			&&=a_e+ a_{f_1} +a_{f_3}, \hspace{2cm}	&&b=a^\mathsf{T}x^{\delta(w)}			&&=a_e+ a_{f_2}+a_{f_4},	\\
		b	&=a^\mathsf{T}x^{\delta(\{v,w_1\})}	&&=a_e+ a_{f_3},				 		&&b=a^\mathsf{T}x^{\delta(\{w,w_2\})}	&&=a_e+ a_{f_4},	\\
		b	&=a^\mathsf{T}x^{\delta(\{v,w_3\})}	&&=a_e+ a_{f_1},				 		&&b=a^\mathsf{T}x^{\delta(\{w,w_4\})}	&&=a_e+ a_{f_2} .	
	\end{alignat*}
	Hence, we have $a_{f_1}=a^\mathsf{T}x^{\delta(\{v\})}-a^\mathsf{T}x^{\delta(\{v,w_1\})}=b-b=0$ and analogously $a_{f_2}=a_{f_3}=a_{f_4}=0$.
\end{proof}
	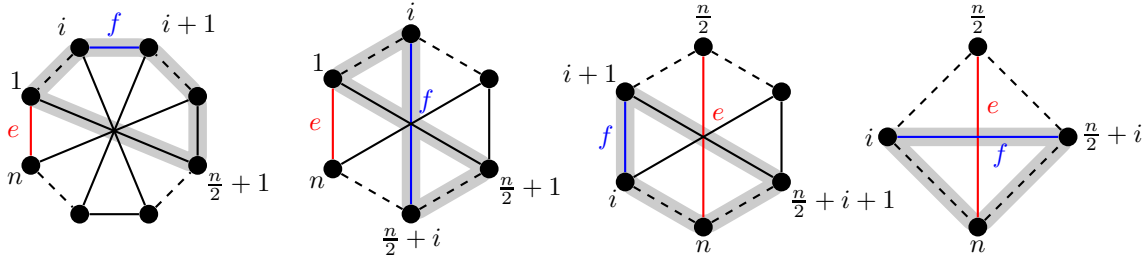
\begin{figure}
		\centering
		\begin{subfigure}{.24\textwidth}
			
			\begin{tikzpicture}[scale=.4]
				\node[circle,fill,scale=.4] (1) at  (1*360/8-360/16 :3){\footnotesize$v$};
				\node[circle,fill,scale=.4, label={[label distance=-1mm]above right:\footnotesize{$i+1$}}] (2) at  (2*360/8-360/16 :3){\footnotesize $v$};
				\node[circle,fill,scale=.4, label={[label distance=-1mm]above left:\footnotesize{$i$}}] (3) at  (3*360/8-360/16 :3){\footnotesize $v$};
				\node[circle,fill,scale=.4, label={[label distance=-2mm]above left:\footnotesize{$1$}}] (4) at  (4*360/8-360/16 :3){\footnotesize $v$};
				\node[circle,fill,scale=.4, label={[label distance=-2mm]below left:\footnotesize{$n$}}] (5) at  (5*360/8-360/16 :3){\footnotesize $v$};
				\node[circle,fill,scale=.4, label={[label distance=-1mm]below:\phantom{\footnotesize{$1$}}}] (6) at  (6*360/8-360/16 :3){\footnotesize $v$};
				\node[circle,fill,scale=.4] (7) at  (7*360/8-360/16 :3){\footnotesize $v$};
				\node[circle,fill,scale=.4,label={[label distance=-2mm]below right:\footnotesize{$\frac{n}{2}+1$}}] (8) at  (8*360/8-360/16 :3){\footnotesize $v$};
				\draw[dashed](1)--(2);
				\draw[blue] (2)--node[above]{\footnotesize $f$}(3);
				\draw[dashed](3)--(4);
				\draw[red] (4)--node[left]{\footnotesize $e$}(5);
				\draw[dashed](5)--(6);
				\draw(6)--(7);
				\draw[dashed](7)--(8);
				\draw(1)--(8);
				
				\draw(1)--(5);
				\draw(2)--(6);
				\draw(3)--(7);
				\draw(4)--(8);
				\begin{pgfonlayer}{bg}
					\draw[gray!40,line width=7pt,line cap=round,rounded corners] (4.center)--(8.center)--(1.center)--(2.center)--(3.center)--(4.center);
				\end{pgfonlayer}
				
			\end{tikzpicture}
		\end{subfigure}
		\hfill
		\begin{subfigure}{.24\textwidth}
			
			\begin{tikzpicture}[scale=.4]
				\node[circle,fill,scale=.4] (1) at  (1*360/6-360/12 :3){\footnotesize$v$};
				\node[circle,fill,scale=.4, label={[label distance=-1mm]above:\footnotesize{$i$}}] (2) at  (2*360/6-360/12 :3){\footnotesize $v$};
				\node[circle,fill,scale=.4, label={[label distance=-2mm]above left:\footnotesize{$1$}}] (3) at  (3*360/6-360/12 :3){\footnotesize $v$};
				\node[circle,fill,scale=.4, label={[label distance=-2mm]below left:\footnotesize{$n$}}] (4) at  (4*360/6-360/12 :3){\footnotesize $v$};
				\node[circle,fill,scale=.4, label={[label distance=-1mm]below:\footnotesize{$\frac{n}{2}+i$}}] (5) at  (5*360/6-360/12 :3){\footnotesize $v$};
				\node[circle,fill,scale=.4,label={[label distance=-2mm]below right:\footnotesize{$\frac{n}{2}+1$}}] (6) at  (6*360/6-360/12 :3){\footnotesize $v$};
				\draw[red] (3)--node[left]{\footnotesize $e$}(4);
				\draw[blue] (2)--node[xshift= 2mm, yshift=3mm]{\footnotesize $f$}(5);
				\draw (1)--(4);
				\draw (3)--(6);
				\draw [dashed] (1)--(2)--(3);
				\draw [dashed] (4)--(5)--(6);
				\draw (1)--(6);
				\begin{pgfonlayer}{bg}
					\draw[gray!40,line width=7pt,line cap=round,rounded corners] (2.center)--(3.center)--(6.center)--(5.center)--(2.center);
				\end{pgfonlayer}
			\end{tikzpicture}
		\end{subfigure}
		\hspace{-.5cm}
		\begin{subfigure}{.24\textwidth}
			
			\begin{tikzpicture}[scale=.4]
				\node[circle,fill,scale=.4] (1) at  (1*360/6-360/12 :3){\footnotesize$v$};
				\node[circle,fill,scale=.4, label={[label distance=-1mm]above:\footnotesize{$\frac{n}{2}$}}] (2) at  (2*360/6-360/12 :3){\footnotesize $v$};
				\node[circle,fill,scale=.4, label={[label distance=-2mm]above left:\footnotesize{$i+1$}}] (3) at  (3*360/6-360/12 :3){\footnotesize $v$};
				\node[circle,fill,scale=.4, label={[label distance=-2mm]below left:\footnotesize{$i$}}] (4) at  (4*360/6-360/12 :3){\footnotesize $v$};
				\node[circle,fill,scale=.4, label={[label distance=-1mm]below:\footnotesize{$n$}}] (5) at  (5*360/6-360/12 :3){\footnotesize $v$};
				\node[circle,fill,scale=.4,label={[label distance=-2mm]below right:\footnotesize{$\frac{n}{2}+i+1$}}] (6) at  (6*360/6-360/12 :3){\footnotesize $v$};
				\draw[blue] (3)--node[left]{\footnotesize $f$}(4);
				\draw[red] (2)--node[xshift= 2mm, yshift=3mm]{\footnotesize $e$}(5);
				\draw (1)--(4);
				\draw (3)--(6);
				\draw [dashed] (1)--(2)--(3);
				\draw [dashed] (4)--(5)--(6);
				\draw (1)--(6);
				\begin{pgfonlayer}{bg}
					\draw[gray!40,line width=7pt,line cap=round,rounded corners] (3.center)--(6.center)--(5.center)--(4.center)--(3.center);
				\end{pgfonlayer}
			\end{tikzpicture}
		\end{subfigure}
		\hfill
		\begin{subfigure}{.24\textwidth}
			
			\begin{tikzpicture}[scale=.4]
				\node[circle,fill,scale=.4, label={[label distance=-1mm]above:\footnotesize{$\frac{n}{2}$}}] (1) at  (1*360/4 :3){\footnotesize$v$};
				\node[circle,fill,scale=.4, label={[label distance=-1mm]left:\footnotesize{$i$}}] (2) at  (2*360/4 :3){\footnotesize $v$};
				\node[circle,fill,scale=.4, label={[label distance=-1mm]below:\footnotesize{$n$}}] (3) at  (3*360/4 :3){\footnotesize $v$};
				\node[circle,fill,scale=.4, label={[label distance=-1mm]right:\footnotesize{$\frac{n}{2}+i$}}] (4) at  (4*360/4 :3){\footnotesize $v$};
				\draw[red] (1)--node[xshift= 2mm, yshift=4mm]{\footnotesize $e$}(3);
				\draw[blue] (2)--node[xshift= 3mm, yshift=-2.5mm]{\footnotesize $f$}(4);
				\draw [dashed] (1)--(2)--(3)--(4)--(1);
				\begin{pgfonlayer}{bg}
					\draw[gray!40,line width=7pt,line cap=round,rounded corners] (2.center)--(3.center)--(4.center)--(2.center);
				\end{pgfonlayer}
			\end{tikzpicture}
		\end{subfigure}
		\caption{A visualization of the set $S^+$ from the proof of \Cref{thm:EdgeIneqWagner}. The subgraph $G[S^+]$ is highlighted.}
		\label{fig:ProofEdgeWagnerSetS}
	\end{figure}
	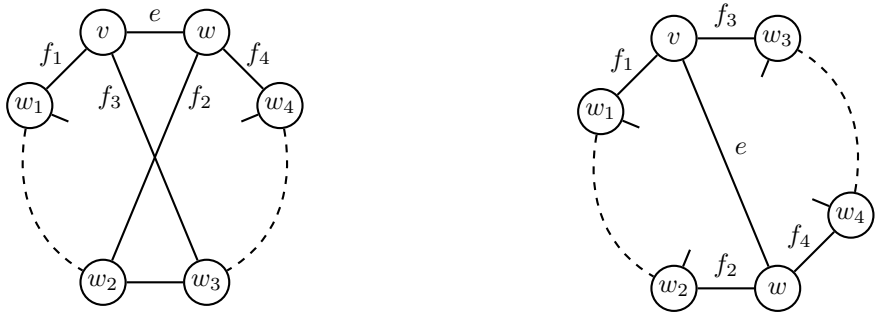
\begin{figure}
		\begin{subfigure}{.5\textwidth}
			\centering
			\begin{tikzpicture}[scale=.6]
				\node[draw,circle,minimum size=.6cm,inner sep=1.0pt] (v) at  (3*360/8-360/16 :3){\footnotesize $v$};
				\node[draw,circle,minimum size=.6cm,inner sep=1.0pt] (w) at  (2*360/8-360/16 :3){\footnotesize $w$};
				\node[draw,circle,minimum size=.6cm,inner sep=1.0pt] (w1) at (4*360/8-360/16 :3){\footnotesize $w_1$};
				\node[draw,circle,minimum size=.6cm,inner sep=1.0pt] (w2) at (6*360/8-360/16 :3){\footnotesize $w_2$};
				\node[draw,circle,minimum size=.6cm,inner sep=1.0pt] (w3) at (7*360/8-360/16 :3){\footnotesize $w_3$};
				\node[draw,circle,minimum size=.6cm,inner sep=1.0pt] (w4) at (1*360/8-360/16 :3){\footnotesize $w_4$};
				\draw (v)--node[above]{\footnotesize $e$}(w);
				\draw (v)--node[xshift=-2mm , yshift=2mm]{\footnotesize$f_1$}(w1);
				\draw (w)-- node[xshift=6mm , yshift=8mm]{\footnotesize$f_2$}(w2);
				\draw (v)-- node[xshift=-6mm , yshift=8mm]{\footnotesize$f_3$}(w3);
				\draw (w) --node[xshift=2mm , yshift=2mm]{\footnotesize$f_4$} (w4);
				\draw (w1)edge [out=-100,in=150,dashed](w2);
				\draw (w4)edge [out=-80,in=30,dashed](w3);
				\draw (w2)--(w3);
				\node (x1) at (4*360/8-360/16 :1.8) {};
				\draw (w1)--(x1);
				\node (x4) at (1*360/8-360/16 :1.8) {};
				\draw (w4)--(x4);
			\end{tikzpicture}
		\end{subfigure}
		\begin{subfigure}{.5\textwidth}
			\centering
			\begin{tikzpicture}[scale=.6]
				\node[draw,circle,minimum size=.6cm,inner sep=1.0pt] (v) at  (3*360/8-360/16 :3){\footnotesize $v$};
				\node[draw,circle,minimum size=.6cm,inner sep=1.0pt] (w3) at  (2*360/8-360/16 :3){\footnotesize $w_3$};
				\node[draw,circle,minimum size=.6cm,inner sep=1.0pt] (w1) at  (4*360/8-360/16 :3){\footnotesize $w_1$};
				\draw (v)--node[above]{\footnotesize $f_3$}(w3);
				\draw (v)--node[xshift=-2mm , yshift=2mm]{\footnotesize$f_1$}(w1);
				\node[draw,circle,minimum size=.6cm,inner sep=1.0pt] (w) at  (7*360/8-360/16 :3){\footnotesize $w$};
				\node[draw,circle,minimum size=.6cm,inner sep=1.0pt] (w2) at  (6*360/8-360/16 :3){\footnotesize $w_2$};
				\node[draw,circle,minimum size=.6cm,inner sep=1.0pt] (w4) at  (8*360/8-360/16 :3){\footnotesize $w_4$};
				\draw (w)--node[above]{\footnotesize $f_2$}(w2);
				\draw (w)--node[xshift=-2mm , yshift=2mm]{\footnotesize$f_4$}(w4);
				
				\draw (v)-- node[xshift=2mm , yshift=2mm]{\footnotesize$e$}(w);
				\draw (w1)edge [out=-100,in=150,dashed](w2);
				\draw (w4)edge [out=80,in=-30,dashed](w3);
				\node (x2) at  (6*360/8-360/16 :1.8){};
				\node (x4) at  (8*360/8-360/16 :1.8){};
				\node (x3) at  (2*360/8-360/16 :1.8){};
				\node (x1) at  (4*360/8-360/16 :1.8){};
				\draw (x1) -- (w1);
				\draw (x2) -- (w2);
				\draw (x3) -- (w3);
				\draw (x4) -- (w4);
				
			\end{tikzpicture}
		\end{subfigure}
		\caption{Sketches of $V_n$ with notations from the proof of \Cref{thm:EdgeIneqWagner}}
		\label{fig:ProofEdgeWagner}
	\end{figure}

Given this result and noticing that each cycle in $V_n$ is interleaved gives rise to the following question:

\begin{restatable}{question}{questtwo}
	Let $G$ be a $3$-connected graph. Does $x_e \leq 1$ define a facet of $\Bond(G)$ for each $e$ that is not contained in a non-interleaved cycle?
\end{restatable}

\section{$\boldsymbol{(K_5-e)}$-Minor Free Graphs}\label{sec:K5-e_Minor}

The focus of this section lies on $(K_5-e)$-minor free graphs.
We prove a linear description of bond polytopes for planar $3$-connected such graphs. Moreover, we present a linear time algorithm for $\MB$ on arbitrary $(K_5-e)$-minor free graphs. We start with a characterization of these graphs.

The \emph{wheel graph} $W_n$ on $n$-nodes is obtained from the cycle $C_n$ by adding a new node $c$ adjacent to each node of $C_n$. We call $c$ the \emph{center node} of $W_n$ and $C_n \subseteq W_n$ the \emph{rim}.
Moreover, we denote the graph shown in \Cref{fig:prism} by $\prism$.

\begin{prop}\cite{Decomposition_K5-e-Free}\label{prop:Wagner_K5-e-free}
	Each maximal $(K_5-e)$-minor free graph $G$ can be decomposed as $G=G_1 \oplus_2 \dots \oplus_2 G_\ell$ where each
	$G_i$ is isomorphic to a wheel graph, $\prism$, $K_3$, or $K_{3,3}$.
\end{prop}

As a consequence , it follows that each $3$-connected $(K_5-e)$-minor free graph is a wheel graph, $\prism$, $K_3$, or $K_{3,3}$. We provide a complete facet description for all planar such graphs (i.e., all but $K_{3,3}$).

\begin{figure}
	\centering
	\begin{tikzpicture}
		\node[circle, fill, scale=.5] (1) at (0,0) {};
		\node[circle, fill, scale=.5] (2) at (0,2) {};
		\node[circle, fill, scale=.5] (3) at (1,1) {};
		\draw (1)--(2)--(3)--(1);
		\node[circle, fill, scale=.5] (a) at (4,0) {};
		\node[circle, fill, scale=.5] (b) at (4,2) {};
		\node[circle, fill, scale=.5] (c) at (3,1) {};
		\draw (a)--(b)--(c)--(a);
		\draw (1)--(a);
		\draw (2)--(b);
		\draw (3)--(c);
	\end{tikzpicture}
	\caption{$\prism$}\label{fig:prism}
\end{figure}

\begin{thm}\label{thm:FacetDescriptionWn}
	Let $G \neq K_{3,3}$ be a $3$-connected $(K_5-e)$-minor-free graph. Then $\Bond(G)$ is completely determined by the following facet-defining inequalities:
	\begin{align*}
		x_e &\geq 0 									\qquad \text{for each edge $e$ that is not contained in a triangle,}\\
		x_e - \sum_{f \in E(C)\setminus \{e\}}x_f &\leq 0	\qquad \text{for each induced cycle $C$ and $e\in E(C)$,}\\
		\sum_{e \in E(C)}x_e &\leq 2 						\qquad \text{for each non-interleaved cycle } C.
	\end{align*}
\end{thm}

\begin{proof}
	For $K_3$, the claim follows directly from the fact that $\CutP(K_3)=\Bond(K_3)$. The description of $\Bond(\prism)$ can be checked by computation. Thus, it remains to prove that $\Bond(W_n)$ is completely defined by the inequalities
	
	\begin{align*}
		x_e -x_f -x_g &\leq 0 							\qquad \text{for each triangle $\{e,f,g\}$ in $W_n$,}			\\
		x_e - \sum_{f \in R\setminus \{e\}}x_f &\leq 0	\qquad \text{for each $e\in R$,}	\\
		x_e +x_f + x_g &\leq 2 							\qquad \text{for each triangle $\{e,f,g\}$ in $W_n$,}			\\
		\sum_{e \in R}x_e &\leq 2,
	\end{align*}
	where $R\subseteq E(W_n)$ denotes the set of rim edges of $W_n$.
	
	By \cite[Corollary 3.10]{OnTheCutPolytope}, the homogeneous inequalities above are the homogeneous facets of $\CutP(W_n)$ and \Cref{thm:Vertices} implies that these are precisely the homogeneous facets of $\Bond(W_n)$.
	
	Let $c$ denote the center node of $W_n$ and $\delta \subseteq E$ be a cut.
	Then, $\delta$ is a bond if and only if $\delta=\emptyset$, $\delta=\delta(c)$ or $|\delta \cap R| =2$.
	
	Let $Q$ denote the polytope given by the above inequalities. Clearly we have $\Bond(W_n) \subseteq Q$. We prove equality of the two polytopes by showing that each vertex of $Q$ is the incidence vector of some bond.

	Note that by \cite[Corollary 3.10]{OnTheCutPolytope}, $Q=\CutP(W_n)\cap \{ \sum_{e \in R} x_e \leq 2 \}$ and the vertices of $\CutP(W_n)$ contained in $\{ \sum_{e \in R} x_e \leq 2 \}$ are exactly the incidence vectors of bonds. Now assume, $Q$ has an additional vertex. Then, this is given as $\relint(F) \cap \{ \sum_{e \in R} x_e = 2 \}$ where $F$ is a $1$-dimensional face of $\CutP(W_n)$. Such a face has to contain one of $x^\emptyset$ and $x^{\delta(c)}$. But by \Cref{prop:AdjacencyInCutP}, in $\CutP(W_n)$ these are only adjacent to incidence vectors of bonds which yields a contradiction.
\end{proof}

Given the previous results, it seems natural to ask, whether the bond polytope of $3$-connected planar graphs is completely described by inequalities associated to cycles and edges. Unfortunately the answer to this question is negative, since already $\Bond(K_5-e)$ has a facet that does not belong to the mentioned class. 

\begin{ex}\label{obs:K5-e}
	Consider $K_5-e$ with the edge labeling as in \Cref{fig:K5-e}. The inequality $x_1+x_2 +x_4+x_5 + x_7 - x_8 - x_9 \leq 2 $ defines a facet of $\Bond(K_5-e)$.~\hfill$\blacktriangleleft$
\end{ex}

\begin{figure}
	\centering
	\begin{tikzpicture}
		\node[circle,fill, scale=0.5] (1) at (360/3 * 1:2cm) {};
		\node[circle,fill, scale=0.5] (2) at (360/3 * 2:2cm) {};
		\node[circle,fill, scale=0.5] (3) at (360/3 * 3:2cm) {};
		\node[circle,fill, scale=0.5] (a) at (0,0) {};
		\draw[red] (a)--node[below, yshift=-1mm]{\small $1$}(1);
		\draw[red] (a)--node[above, yshift=1mm]{\small $2$}(2);
		\draw (a)--node[above left,, yshift=-1mm]{\small $3$}(3);
		\draw[red] (1)--node[left, xshift=1mm]{\small $7$}(2);
		\draw[red] (2)--node[right, yshift=-1mm]{\small $9$}(3);
		\draw[red] (3)--node[right, yshift=1mm]{\small $8$}(1);
		\node[circle,fill, scale=0.5] (b) at (-4,0) {};
		\draw[red] (b)--node[above]{\small $4$} (1);
		\draw[red] (b)-- node[below]{\small $5$}(2);
		\draw (b)edge[in=-90, out=-90, distance =3cm] node[below]{\small $6$}(3);
		
	\end{tikzpicture}
	\caption{$K_5-e$ on the edge set $E=\{1,\dots 9\}$. Red edges are those of the support graph of the inequality from \Cref{obs:K5-e}.}
	\label{fig:K5-e}
\end{figure}
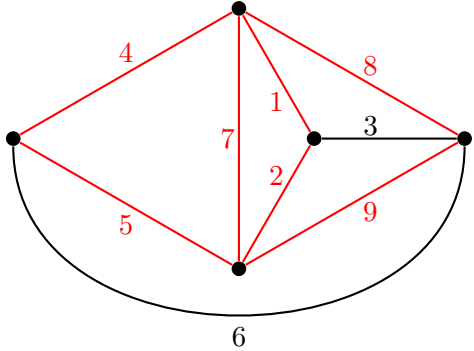

We close this section by presenting a linear time algorithm for $\MB$ on $(K_5-e)$-minor free graphs. For details on tree-width and parameterized algorithms see, e.g., \cite{downey2013fundamentals}.

\begin{prop}\cite{computing_largest_bond, MaxBondMaxConnectedCut}\label{prop:MaxBond_bounded_treewidth}
	Given a nice tree decomposition of $G$ with width $k$, $\MB$ can be solved on $G$ in time $2^{\mathcal{O}(k \log(k))}\times |V(G)|$.
\end{prop}

\begin{thm}
	Given a $(K_5-e)$-minor free graph $G=(V,E)$ with $|V|=n$, \MB  can be solved on $G$ in time $\mathcal{O}(n)$.
\end{thm}
\begin{proof}
	By \Cref{prop:Wagner_K5-e-free} and the fact that $\prism$, $K_3$, $K_{3,3}$ are of constant size, and $|E(W_n)| = 2 (|V(W_n)|-1)$, we have $|E(G)| \in \mathcal{O}(n)$.
	Using \Cref{thm:Reduction_3connectivity_Components}, we can restrict ourselves to $3$-connected $(K_5-e)$-minor free graphs by only $\mathcal{O}(n)$ additive effort. 
	By \Cref{prop:Wagner_K5-e-free} these graphs are wheel graphs, $\prism$, $K_3$, and $K_{3,3}$.
	
	Since we can solve $\MB$ in constant time on $\prism$, $K_3$, and $K_{3,3}$, it only remains to prove that $\MB$ can be solved in time $\mathcal O(n)$ on wheel graphs.
	Denoting the center of $W_n$ by $c$ and the rim nodes by $v_1,\dots,v_n$, it is straight forward to verify that a nice tree decomposition of $W_n$ with width $3$ is given by the bags $\{c\}, \{cv_1\}, \{cv_1v_2\}, \{cv_1v_2v_3\}, \{cv_1v_3\}, \{cv_1v_3v_4\}, \{cv_1v_4\}, \{cv_1v_4v_5\},\dots,\{cv_1v_{n-1}\},$\\$\{cv_1v_{n-1}v_n\}$.
	Using this tree decomposition, \Cref{prop:MaxBond_bounded_treewidth} yields the claim.		
\end{proof}

Note that although the above algorithm has asymptotically linear runtime, the runtime is dependent on large constants.
Since the presented tree decomposition for wheel graphs is in fact even a path decomposition and wheel graphs are of special simple structure for this measure, it should certainly be possible to improve on the constant quite a bit.
Although this might yield a more practical algorithm, this would be out of scope for this work.

\section{Conclusion}
We have introduced bond polytopes and investigated the relation of these to cut polytopes. Then, we have studied the effect of graph-operations on facets of bond polytopes. We have presented a reduction of \MB to $3$-connected graphs. Moreover, we have started an investigation of cycle- and edge inequalities for bond polytopes and derived a family of facet-defining inequalities for bond polytopes. Finally, we have presented a linear time algorithm for \MB on $(K_5-e)$-minor free graphs as well as a linear description for all $3$-connected planar such graphs.  

Recall the open problems from \Cref{sec:interleaved_cycles_and_edges}:
	\questone*
	\questtwo*

On the algorithmical side, we have seen the importance of clique sums. 
Considering cut polytopes, for $k \leq 3$ we can derive a linear description of $\CutP(G_1 \oplus_k G_2)$ given linear descriptions of $\CutP(G_1)$ and $\CutP(G_2)$.
While we have seen how to handle $1$- and $2$-sums in algorithms for \MB, we could not mirror this into the world of bond polytopes.
As a result this would for example yield a linear description for arbitrary $(K_5-e)$-minor free graphs.
Therefore, the following question arises:

\begin{question}
	Given a graph $G=G_1 \oplus_k G_2$ and linear descriptions of $\Bond(G_1)$ and $\Bond(G_2)$. Can we derive a linear description of $\Bond(G)$, at least for $k=1,2$?
\end{question}
As a first step, one may investigate how facets of $\Bond(G_1)$ and $\Bond(G_2)$ can be combined to obtain facets of $\Bond(G)$.
\bibliographystyle{alpha}
{\footnotesize \bibliography{bibliography}}

\end{document}